\documentclass[reqno,11pt]{amsart}
\usepackage{cases,comment}
\usepackage{mathrsfs}
\usepackage[T1]{fontenc}
\usepackage{mathrsfs}
\usepackage{amsmath,latexsym,amssymb,amsfonts,amsbsy, amsthm}
\usepackage{bm}
\usepackage[usenames]{color}
\usepackage{xspace,colortbl}
\usepackage{epsfig}
\usepackage{graphicx}
\usepackage{subfigure}
\usepackage{amsmath,amsfonts,amsthm,amssymb,amscd}
\usepackage{color}
\usepackage[title,titletoc,toc]{appendix}
\usepackage{bigints}
\usepackage{cite}
\usepackage{hyperref}
\allowdisplaybreaks

\newcommand{\be}{\begin{equation} }
\newcommand{\ee}{\end{equation}}
\newcommand{\bee}{\begin{equation*} }
\newcommand{\eee}{\end{equation*}}
\newcommand{\bse}{\begin{subequations}}
\newcommand{\ese}{\end{subequations}}

\newcommand{\LN}{\left\|}
\newcommand{\RN}{\right\|}

\newcommand{\R}{\mathbb{R}}

\theoremstyle{plain}
\newtheorem{theorem}{Theorem}[section]
\newtheorem{corollary}[theorem]{Corollary}

\newtheorem{lemma}{Lemma}[section]
\newtheorem{proposition}{Proposition}[section]
\theoremstyle{remark}
\newtheorem{remark}{Remark}[section]

\theoremstyle{definition}

\numberwithin{equation}{section}

\title[ Asymptotic stability  for the $b$-family of equations]{ Asymptotic stability  of solitary waves for the $b$-family of equations}

\author[J. Wu]{Jun Wu}
\address[Jun Wu]{School of Mathematics and Statistics, Henan University
475004, China}\email{wujun@henu.edu.cn}
\author[Y. Liu]{Yue Liu}
\address[Y. Liu (corresponding author]{Department of Mathematics, University of Texas at Arlington, TX 76019, USA}\email{yliu@uta.edu}
\author[Z. Wang]{Zhong Wang}
\address[Zhong Wang]{School of Mathematics, Foshan University, 528000,  China}\email{wangzh79@fosu.edu.cn}

\date{}

\begin{document}
\thispagestyle{empty}

\begin{abstract}

We establish the asymptotic stability of lefton solutions-exponentially localized stationary solitary waves-for the $b$-family of equations with positive momentum density in the regime $b < -1$. Unlike the completely integrable Camassa-Holm $(b=2)$ and Degasperis-Procesi $(b=3)$ cases, this parameter range lies outside integrability and exhibits distinct nonlinear dynamics. Our analysis adapts the Martel-Merle framework for generalized KdV equations to the nonlocal, non-integrable structure of the $b$-family of equations. The proof combines a nonlinear Liouville property for solutions localized near leftons with a refined spectral analysis of the associated linearized operator. These results provide the first rigorous asymptotic stability theory for leftons in the non-integrable $b$-family of equations.


\end{abstract}

\maketitle

\maketitle

\section{Introduction}\label{sec1}
We investigate the one-dimensional $b$-family of nonlinear evolution equations, which models the competition between convection and stretching mechanisms in fluid dynamics within the context of nonlinear wave propagation \cite{DHH1,DGH}. In the absence of viscosity and linear dispersion, the $b$-family of equations  takes the form
\begin{equation}\tag{$b$-family}\label{main}
u_t - u_{xxt} + (b+1)uu_x = b u_x u_{xx} + u u_{xxx}, \quad u(t,x)\in \mathbb{R}, \ \ (t,x)\in \mathbb{R}\times\mathbb{R},
\end{equation}
where $u=u(t,x)$ denotes the fluid velocity field, $x\in \mathbb{R}$ is the spatial variable, $t\in \mathbb{R}$ is time, and $b\in \mathbb{R}$ is a dimensionless parameter representing the relative strength of stretching versus convection in the dynamics of nonlinear waves. The $b$-family of equations can be viewed as a nonlinear dispersive generalization of shallow-water models. Its derivation traces back to asymptotic reductions of the Korteweg-de Vries (KdV) equation, obtained via a Kodama transformation followed by a Galilean transformation \cite{DHH1,DGH}.
A remarkable feature of \eqref{main} is that it unifies, within a single parametric framework, several well-known nonlinear wave equations. Among these, two distinguished cases are completely integrable:
The Camassa-Holm (CH) equation ($b=2$), first derived in \cite{CH,CHH} as an approximation to the incompressible Euler equations in shallow water, modeling unidirectional wave propagation. The CH equation admits a Lax pair, bi-Hamiltonian structure, and peaked soliton (peakon) solutions.
The Degasperis-Procesi (DP) equation ($b=3$), identified in \cite{DP} via the method of asymptotic integrability applied to third-order dispersive conservation laws. Like the CH equation, the DP equation is integrable and supports peakon solutions, but with different stability and interaction properties.
Beyond these two exceptional cases ($b=2,3$), the $b$-family of equations \eqref{main} is not completely integrable. Nevertheless, it remains an important model for exploring nonlinear wave phenomena in fluid mechanics, providing a bridge between integrable systems, approximate models of the Euler equations, and more general non-integrable wave dynamics.

Most studies on \eqref{main} have concentrated on its alternative formulation. This reformulation arises naturally by applying the Helmholtz operator to $u$, which not only simplifies the analytical structure of the equation but also reveals its Hamiltonian formulation and geometric interpretation. In this form, equation \eqref{main} can be equivalently rewritten as
\begin{equation}\label{bfe}
    m_t+um_x+bu_xm=0\quad\quad\text{with}\quad m:=u-u_{xx},
\end{equation}
which is a  transport equation for the momentum density $m$.
In this setting of the transport equation, although the $b$-family equation is not integrable for $b\neq2,3$, there are the following three conserved quantities along the flow of \eqref{bfe} for $b\neq0,1$,
\begin{align}\label{enc}
    E=\int_{\mathbb{R}}m\,\mathrm{d}x,~~~~~~~F_1=\int_{\mathbb{R}}m^{\frac{1}{b}}\,\mathrm{d}x\quad
    \text{and}
    \quad F_2=\int_{\mathbb{R}}m^{-\frac{1}{b}}\left(\frac{m_x^2}{b^2m^2}+1\right)\,\mathrm{d}x.
\end{align}
The conservation law formulation of the $b$-family of  equations \eqref{bfe} is derived by means of the Helmholtz operator, through which \eqref{bfe} can be recasted into the following form
\[ m_t+\big(\frac{b-1}{2}(u^2-u_x^2)+um\big)_x=0.\]
There is a skew-symmetric operator, which shows up in the abstract Hamiltonian system
\begin{equation}\label{AHS}
m_t=\frac{1}{b-1}\mathcal{B}\frac{\delta E}{\delta m},
\end{equation}
where the functional $E$ is defined in \eqref{enc}, and the skew-symmetric operator
\begin{equation}\label{ob}
 \mathcal{B}:=\mathcal{B}(m)=-\left(b m \partial_x+m_x\right)\left(\partial_x-\partial_x^3\right)^{-1}\left(b m \partial_x+(b-1) m_x\right)
\end{equation}
satisfies the Jacobi identity \cite{HH05}. In fact, the Hamiltonian \eqref{AHS} is equivalent to the $b$-family of equations \eqref{bfe} by employing the Legendre transformation for $b\neq1.$ The Hamiltonian operator $\mathcal{B}$ is still valid in the case $b = 1,$ but we should replace $\frac{1}{b-1}E$ with
$$\int m\log m\,\mathrm{d}x.$$

Contributing to the fascination with the $b$-family of equations is the fact that it possesses many forms of special solutions. For the zero asymptotic value (the asymptotic value is the
traveling wave solution at infinity), different values of $b$ significantly alter the behavior of the solution to \eqref{main}, the work of Holm and Staley \cite{HS1,HS2} exhibited that
assigning different initial data and values of $b$ yields different numerical simulation results for the solution of \eqref{main}. They distinguished that there are three parameter regimes for the key point at $b=1$ and $b=-1$. The initial data separates asymptotically into a series of peaked traveling
waves in the Peakon regime with $b>1$ and a number of exponentially localized stationary solitary waves with the $\operatorname{sech}$ shape called "leftons" for $b<-1.$
In the Ramp-cliff regime with $-1<b<1,$ the asymptotic behavior of the solution is analogous to the combination of a "ramp" solution (proportional to $x/t$) with an exponentially-decaying tail ("cliff"). In a certain sense, the separation behaviors of peakon regime and lefton regime are consistent with the soliton resolution conjecture.

The derivation of \eqref{main} was carried out according to the Galilean transform making the linear dispersion terms absorb away, which loses the balance between nonlinear steepening and linear dispersion. However, the balance of nonlinear and non-local terms in \eqref{main} can still produce peaked solitary waves called peakons with any  $b > 1,$
\begin{equation}\label{peakon}
u(t,x)=ce^{-|x-ct|}.
\end{equation}
The orbital stability of peakons for the CH equation ($b=2$) \cite{CM,CS} in the energy space $H^1(\mathbb{R})$ and for the DP equation $(b=3)$ \cite{LL} in the energy space $L^2(\mathbb{R})$ was known by using the energy functional to establish two integral relations. Indeed, since the Cauchy problem of the $b$-family of equations is ill-posed in $H^s(\mathbb{R})$ for $s<\frac{3}{2}$ with $b>1,$ the orbital stability results remain valid while the local solutions exist. It was shown in \cite{LM, LM19} that the peakons \eqref{peakon} of  the CH equation ($b=2$) and  the DP equation $(b=3)$ are asymptotically stable in $H^1(\R)$ with a momentum density that is a non negative finite measure. The main point was the proof of a rigidity result for uniformly almost localized solutions of
\eqref{main} for all $b\geq1$ in this class. For $b<1,$ using numerics, in their pioneering works Holm and Staley \cite{HS1,HS2} showed that the peakons of the $b$-family equation are unstable. Subsequent to the numerical simulations, thorough spectral analysis of the linearized operator and numerical analysis in \cite{LP,CPKL} verifies the instability conjecture. In addition, they also numerically indicated that the stationary solitary wave solution only exist in the parameter regime $b<-1$.

A single stationary solitary wave solution, called a lefton, of \eqref{main} is a special solution of the form
\begin{equation}\label{ul}
    q(x)= A \left(\cosh{\nu(x-x_*)}\right)^{-\frac{1}{\nu}},\quad \nu=-\frac{b+1}{2}>0,
\end{equation}
where the arbitrary constants $x_*$ and $A>0$ represent the position parameter and amplitude, respectively. Meanwhile, in terms of momentum density,  it is easy to verify that
$$u=p*m=\int_{\mathbb{R}}p(x-y)m(y)dy,$$
where $p(x)=\frac12e^{-|x|}$ is the Green function of the Helmholtz operator $1-\partial_x^2.$ The lefton corresponding to the the momentum density $m$ has the form
\begin{equation}\label{ml}
    Q(x)= (1 - \partial_x^2) q(x) = A\frac{1-b}{2}\left(\cosh{\nu(x-x_*)}\right)^{\frac{b}{\nu}},\quad \nu=-\frac{b+1}{2}>0, \quad A>0,
\end{equation}
for $b<-1,$ and it retains a shape similar to that of KdV soliton.
The relationship between $q$ and  $Q$ can be characterized by $q^bQ=\frac{1-b}{2}A^{b+1}$.
\begin{remark}
    Observing the three conserved quantities of the $b$-family of equations, we can see that within the real line, the functional $F_1$ diverges for $b<0$, and the functional $F_2$ is only meaningful for  $m>0$.
\end{remark}


It is noted that the results of  global well-posedness  of the Cauchy problem for \eqref{main}  and the orbital stability  form a necessary step in our study of asymptotic stability, thereby providing a rigorous framework for  investigating the large time dynamical behavior of the solutions. Before accurately describing orbital stability and global well-posedness, we should provide an appropriate function space. In fact, the second derivative of $-E+kF_2$ needs to be discussed during the process of proving orbital stability. For smooth $m$ and $v$, the second variation of $-E+kF_2$ is given by
$$\delta^2\left(-E+kF_2\right)(m,v)=k\int_{\mathbb{R}}\left(T_1v_x^2+T_2v^2\right)\,\mathrm{d}x,$$
where
\begin{eqnarray}
    k&=&(A\frac{1-b}{2})^{\frac{1}{b}+1},\quad T_1=\frac{2}{b^2}m^{-\frac{1}{b}-2},\label{Lag multiplier}\\
    T_2&=&\frac{m^{-\frac{1}{b}-4}}{b^4}\left(2 b(1+2 b) m_{x x} m-(1+2 b)(1+3 b) m_x^2+b^2(1+b) m^2\right).\nonumber
\end{eqnarray}
The second variation of $-E+kF_2$ near $m=Q$ is
$$\delta^2\left(-E+kF_2\right)(Q,v)=\frac{2k}{b^2}\int_{\mathbb{R}}Q^{-\frac{1}{b}-2}\left( v_x^2-b(b+1)v^2\right)\,\mathrm{d}x.$$
In accordance with the formulation of second variation at $m=Q$, we define the weight
\begin{equation}\label{wei}
    \alpha:=Q^{-\frac{1}{b}-2}=\left(A \frac{1-b}{2}\right)^{-\frac{1}{b}-2}\left(\cosh \nu\left(x-x_*\right)\right)^{-\frac{2 b+1}{\nu}},
\end{equation}
which generates the weighted Sobolev space $H_{\alpha}^1:=H^1\left(\mathbb{R};\alpha\,\mathrm{d}x\right)$ with the inner product
$$\left(u,v\right)_{\alpha}=\int_{\mathbb{R}}u v\alpha\,\mathrm{d}x,~~~\left\langle u,v\right\rangle_{\alpha}=\int_{\mathbb{R}}\left(uv+u_xv_x\right)\alpha\,\mathrm{d}x.$$
It is obvious that $\alpha$ is  bounded uniformly away from zero and $H_{\alpha}^1$ is a subspace of Hilbert space $H^1.$  The space defined in \cite{HL}
$$\mathcal{Z}:=\left\{f \in H_\alpha^1 (\mathbb{R})\mid f=O\left(Q\right) \quad \text { as } \quad|x| \rightarrow \infty\right\} \subset H_{\alpha}^1,$$
was taken into account such that $F_2$ exists. The norm of $\mathcal{Z}$ is defined by $\|f\|_{\mathcal{Z}}=\|f\|_{H^1_{\alpha}}+K_{_{\mathcal{Z}}},$
where $\|f\|_{H^1_{\alpha}}=\sqrt{\langle f,f\rangle_{\alpha}}$ and
$$K_{_{\mathcal{Z}}}=\sup_{x\in\mathbb{R}}\frac{|f(x)|}{Q(x)}.$$
In the framework of the weighted space, the $b-$family equation \eqref{bfe} is known, \cite{HL}, to be globally well-posed in $\mathcal{Z}$ with any positive initial data $m(\cdot,0)\in \mathcal{Z}$ and $F_2<\infty$.
The global well-posedness of the cauchy problem to \eqref{main} in $H^s(\mathbb{R}),~s>\frac32$ and nonnegative $m_0=(1-\partial_x^2)u_0 $ was proved in \cite{EY,CE}. In addition, they also established a unique global
weak solution $u\in W^{1,\infty}(\mathbb{R}_+,\mathbb{R})\cap L^{\infty}(\mathbb{R}_+,H^1(\mathbb{R}))$ such that $u(\cdot,t)-u_{xx}(\cdot,t)\in\mathrm{M}_+,$ where $\mathrm{M}_+$ is the set of non-negative finite Radon measures on $\mathbb{R}$.

The purpose of the present paper is to establish the asymptotic stability of the smooth stationary solutions in the $b$-family of equations for any $b<-1$. Spectral analysis of linearized operators around the leftons is essential for proving both orbital stability and asymptotic stability. In the following, we denote by the operator $\mathcal{L}$ which is related to the linearization of \eqref{main} around the leftons \eqref{ul} and \eqref{ml}, where
\begin{equation}\label{ol}
 \mathcal{L}:=\frac{\delta^2(-E+kF_2)}{\delta m^2}(Q)=k\left(-\partial_x\left(T_0\partial_x\right)+H_0\right),
\end{equation}
where $T_0=\frac{2\alpha}{b^2},~H_0=-\frac{2(b+1)\alpha}{b}.$ It was shown in \cite{HL} that the operator $\mathcal{L}$ defined in \eqref{ol} has only one negative eigenvalue, its kernel is one dimensional, and the continuous spectrum of which is bounded below away from zero on the $H_{\alpha}^1$.
Moreover, the eigenvalue problem associated with $\mathcal{L}$ is $\mathcal{L}f=\lambda\alpha f,$ where $\lambda$ is the eigenvalue corresponding to the eigenfunction $f.$
Indeed, the spectral property of the operator $\mathcal{L}$ is significant for us to resolve asymptotic stability. Thanks to the change of variables $\tilde{f}=\sqrt{\alpha}f,$ the eigenvalue problem associated to $\mathcal{L}$ is rewritten as
\begin{equation}\label{epn}
    H\tilde{f}=\lambda\tilde{f},\quad H=\frac{2k}{b^2}\left(-\partial_x^2+\tilde{H}_0\right),
\end{equation}
where $$\tilde{H}_0=\frac{1 }{2\sqrt{\alpha}}\left(\frac{\alpha_x}{\sqrt{\alpha}}\right)_x-2b(b+1) =\frac{1}{4}-\frac{b(1+2b)}{4k}Q^{\frac{1}{b}+1}=\frac{1}{4}-\frac{b(1+2b)}{4}\mathrm{sech}^2\nu(x-x_*).$$
It is clear that the operator $H$ is similar to the linearized operator of the gKdV equation.
The properties of the operator $\mathcal{L}$ can be verified by organizing the properties of the operator $H$ and the change of variables. Recall that $H$ is a classical operator (see Titchmarsh \cite{T46}),  we then collect the following facts about the operator $H$.
\begin{lemma}\label{SPH1}
 The operator $H$ satisfies the following  properties in $H_{\alpha}^1$:
\begin{itemize}
  \item[(1)]  $H$ is self-adjoint and
\begin{equation}\label{cs}
\sigma_{e s s}(H)=\left[\frac{k}{2b^2},+\infty\right).
\end{equation}
 \item[(2)] The kernel of $H$ is one-dimensional and spanned by
 \begin{equation}\label{ker}
   \operatorname{span}\left\{\sqrt{\alpha}\partial_{x} Q\right\}.
 \end{equation}
\item[(3)]  $H$ has exactly one simple negative  eigenvalue $\lambda_0=-k\left(\frac12-\frac{1}{2b^2}\right)$ associated to the eigenfunction $\sqrt{\alpha}Q^{\frac{1}{2b}+\frac{3}{2}}.$
\item[(4)]  \begin{equation}\label{ef1}
H\left(\sqrt{\alpha}Q\right)=-\frac{(1+b)}{b}\sqrt{\alpha}Q^{\frac{1}{b}+2}=-\frac{(1+b)}{b}\sqrt{\alpha}\alpha^{-1}.
 \end{equation}
\item[(5)]  \begin{equation}\label{ef2}
H\left(\sqrt{\alpha}Q^{2}\right)=\frac{2k(1-b)}{b}\sqrt{\alpha}Q^{2}+\frac{2(b-1)}{b}\sqrt{\alpha}Q^{\frac{1}{b}+3}.
\end{equation}
\item[(6)]  \begin{equation}\label{ef3}
H\left(\frac{2b}{b+1}\sqrt{\alpha}Q+x\sqrt{\alpha}\partial_xQ\right)=\frac{2}{b}\sqrt{\alpha}Q.
\end{equation}
    \end{itemize}
\end{lemma}
It is evident that we benefit significantly from the change of variables $\tilde{f}=\sqrt{\alpha}f$, which allows the properties of $H$ to be generalized to operator $\mathcal{L}$.
It is easy to verify the spectral property of the operator $\mathcal{L}$ and we can obtain that $\mathcal{L}$ has exactly one simple negative  eigenvalue $-k\left(\frac12-\frac{1}{2b^2}\right)$ associated with the eigenfunction $ Q^{\frac{1}{2b}+\frac{3}{2}}$, that is,
\begin{equation}\label{efl}
\mathcal{L}\left(Q^{\frac{1}{2b}+\frac{3}{2}}\right)=-k\left(\frac12-\frac{1}{2b^2}\right)\alpha Q^{\frac{1}{2b}+\frac{3}{2}}.
\end{equation}
Moreover, the kernel of the operator $\mathcal{L}$ is one dimensional and spanned by the function $\partial_xQ.$

$Notation.$ Throughout this paper, we denote by $a\lesssim b$ the relation that $a\leq Cb$ for some positive constant $C.$

\subsection{Statement of the main results}
The orbital stability of leftons in $H^1(\mathbb{R})$ with $b<-1$ is given in the following sense.
\begin{proposition}[\cite{HL}]\label{os}
Let $b<-1$. There exist $C > 0 $ and $ \epsilon_0>0$ such that if $ m_0(x) = (1 - \partial_x^2) u_0(x) > 0, \forall x \in \mathbb{R},$ $ F_2 < \infty $ and $m_0 \in \mathcal{Z}$ with
$\left\|m_0-Q\right\|_{\mathcal{Z}}\leq\epsilon\leq\epsilon_0,$ then the corresponding solution $m$ of \eqref{bfe} with the initial value $m_0$ satisfies
$$\sup_{t \in \mathbb{R}}\inf_{y\in\mathbb{R}}\left\|m(\cdot,t)-Q(\cdot-y)\right\|_{H^1(\mathbb{R})}\leq C\epsilon.$$
\end{proposition}
\begin{remark}
    There is a parameter $a\in\mathbb{R}$ appears in the steps proving stability. In fact, the decomposition $\varepsilon=m(\cdot+s(m))-(1+a)Q$ was discussed in \cite{HL} such that $\varepsilon\in H_{\alpha}^1(\mathbb{R})$ satisfying $$\left(\varepsilon,\partial_xQ\right)_{\alpha}=(\varepsilon,F_2'(Q))_{\alpha}=0.$$ 
\end{remark}

We now present our main result on the asymptotic stability of the leftons.

\begin{theorem}\label{as}
Let $b<-1$ and $m_0\in \mathcal{Z}$ with  $ m_0(x) = (1 - \partial_x^2) u_0(x) > 0, \forall x \in \mathbb{R}$ and $ F_2 < \infty $ and let $m$ be the unique global solution of the Cauchy problem associated with $b$-family of equations \eqref{bfe} emanating from initial data $m_0$. There exists $\epsilon_0>0$ such that, if $0<\epsilon<\epsilon_0$ and
\begin{align}\label{osa}
\left\|m_0-Q\right\|_{\mathcal{Z}}\leq\epsilon,
\end{align}
then there exists a function $ \rho(t)\in C^1(\mathbb{R})$ with $\lim_{t\to +\infty}\rho^{\prime}(t)=0$ satisfying
\begin{align}\label{wc}
&m(\cdot,t) \rightharpoonup \gamma Q\left(\cdot-\rho(t)\right) \quad \text { in }\quad H^1(\mathbb{R}),\ \gamma=1+a,\ a\in\R,
\end{align}
as $t \rightarrow+\infty$. Moreover, it also holds that
\begin{equation}\label{asympt lim}
\lim_{t\rightarrow+\infty}\left\|m(\cdot,t)-\gamma Q(\cdot-\rho(t))\right\|_{H^1(x>\beta t)}=0,
\end{equation}
for some $\beta>0$.
\end{theorem}

The structure of the equation in the vicinity of leftons allows for a detailed analysis of the long-time asymptotic behavior of solutions. This analysis relies on refined estimates of the dispersive properties of the linearized operator around the lefton profile. Consequently, a precise understanding of the stability properties of the leftons is essential.

The proof of Theorem \ref{as} is reduced to a rigidity property of
\eqref{bfe} close to the leftons. More precisely, we show that asymptotic stability of leftons is true if
the following nonlinear Liouville-type theorem holds.

\begin{theorem}[Nonlinear Liouville theorem]\label{nlt}
Let $b<-1,$ $m_0\in \mathcal{Z}$, and let $m$ be the unique global solution of the $b$-family of equations \eqref{bfe} emanating from the initial data $m_0 > 0, \forall x \in \mathbb{R} $ with $ F_2 < \infty$. There exists positive numbers $A_0,\epsilon_0>0$ such that, if
$
\left\|m_0-Q\right\|_{\mathcal{Z}}\leq\epsilon\leq\epsilon_0,
$
and the solution $m$ satisfies for a function $ \rho(t)\in C^1(\mathbb{R})$ and some positive constant $\epsilon>0$
\begin{align}\label{lcuv}
\int_{|x|>A_0}m^{-\frac1b}(x+\rho(t),t)\,\mathrm{d}x\leq\epsilon,
\end{align}
then there exists $r\in\mathbb{R}$ such that
\begin{align}\label{nltt}
m(x,t)=\gamma Q(x-r).
\end{align}
\end{theorem}

\begin{remark}
We call the solution $m$ of  \eqref{bfe}  $L^{-\frac1b}$-localized if it satisfies the condition \eqref{lcuv}, which is quite different from the CH equation \cite{LM} (called $Y$-almost localized) and gKdV cases \cite{MM3} ($L^{2}$-localized). Compared with the work in  \cite{LM}, which verifies the nonlinear Liouville property of \eqref{bfe} for all $b\geq1$, we  fill the gaps in nonlinear Liouville property of \eqref{bfe} for $b<-1$.
\end{remark}

In order to prove the nonlinear Liouville property (Theorem \ref{nlt}), one needs to
reduce the proof to a similar property on a related linear problem. Indeed, since we want
to prove a result in a neighborhood of the leftons $Q$ by passing to the limit
on renormalized problems, we show that it is sufficient to classify a related linearized
problem. In the following, we prove a linear Liouville theorem for the localized solutions of the linearized $b$-family of equations around the lefton.

\begin{theorem}[Linear Liouville theorem]\label{llt}
Let $b<-1$ and  $v\in C\left(\mathbb{R},H_{\alpha}^1\right)\cap L^{\infty}\left(\mathbb{R},H_{\alpha}^1\right)$ be the solution of
\begin{equation}\label{l2c}
\partial_tv=\frac{1}{1-b}\mathcal{B}(Q)\mathcal{L}v,
\end{equation}
where $\mathcal{B}$ and $\mathcal{L}$ are defined in \eqref{ob} and \eqref{ol}, respectively. If there exists a constant $\epsilon>0$ such that
\begin{align}\label{lceta}
\int_{|x|>A_0}v^2(x,t)\alpha(x)\,\mathrm{d}x\leq\epsilon,
\end{align}
then there exist constants $a_0,b_0\in\mathbb{R}$ such that
\begin{align}\label{lltt}
v(x,t)=a_0\partial_xQ+b_0Q,
\end{align}
for all $t\in\mathbb{R}.$
\end{theorem}
\begin{remark}

Owing to the non-commutativity of the operators $\mathcal{B}$ and $\mathcal{L}$, the evaluation of $\mathcal{B}(Q)\mathcal{L}$ and $\mathcal{L}\mathcal{B}(Q)$ is nontrivial. The detailed computations are deferred to the Appendix. Unlike in the gKdV equation, where the skew-symmetric operator in the Hamiltonian structure is the local derivative $\partial_x$, the operator $\mathcal{B}$ defined in \eqref{ob} within the Hamiltonian structure of \eqref{bfe} is highly nonlocal. Consequently, after acting on the linearization operator, the resulting expressions cannot be reduced back to the linearization operator by integration by parts. To address this difficulty, in Section~\ref{s2} we adopt the strategy of Kenig and Martel \cite{KM}, introducing directly the standard $L^2$ inner product associated with the linearization operator together with a suitable time-dependent weight function in order to establish the desired uniform boundedness.
\end{remark}

In  proving Theorems~\ref{as}, \ref{nlt}, and \ref{llt}, we establish an asymptotic stability result for the $b$-family of equations in the regime $b < -1$. This fills an existing gap in the literature concerning the asymptotic behavior of the $b$-family across different parameter ranges. Compared with previously studied regimes, the case $b < -1$ introduces new analytical difficulties.
The asymptotic stability problem has a long history, beginning with the seminal works of Soffer, Pego, and Weinstein \cite{SW,SW1,SW2,PW}, who developed a framework in weighted Sobolev spaces to prove the stability of ground states. Building on a different perspective, Martel and Merle \cite{MM3} introduced a novel methodology that combines monotonicity formulae with a Liouville-type property, thereby avoiding the use of weighted spaces. This approach not only removes the need for additional a priori assumptions but also extends naturally to multi-soliton configurations \cite{MM7}. It was first applied to prove the asymptotic stability of solitary waves for the gKdV equation in the energy space $H^1$ \cite{MM3}, and has since been extended to a broad class of nonlinear dispersive models \cite{BGS, CMPS, KM, LM, CLY, CLLW}. In particular, by combining the Martel-Merle framework with the complete integrability of the dispersive Camassa-Holm equation, the asymptotic stability of smooth solitons and multi-solitons in $H^1$ was established in \cite{CLLW}, relying crucially on its bi-Hamiltonian structure and a Liouville-type theorem near solitons.
A key ingredient of the Martel-Merle method is the derivation of monotonicity formulae, which yield decay estimates through direct differentiation. However, for the $b$-family with $b < -1$, the dispersive term is absorbed during this derivation process, obstructing the construction of monotonicity estimates and virial-type bounds. This contrasts sharply with the gKdV equation, where such absorption does not occur and negative weighted norms are not inherent to the formulation. To overcome this difficulty, we establish monotonicity properties for solutions in the spatial $x$-direction but taken backward in time. In this setting, the negative weighted $H^1_\alpha$ functional arises not from the dispersive term but from the intrinsic structure of the weight function built into the monotonicity formula. This constitutes the first major challenge in the asymptotic stability analysis of the $b$-family of equations, reflecting the absence of a linear dispersive mechanism.

The second key ingredient in establishing asymptotic stability is the spectral analysis of the linearized operator around the leftons (see Lemma~\eqref{SPH1}). As noted earlier, the available spectral results are incomplete. A useful observation is that the change of variables $\tilde{f} = \sqrt{\alpha}, f$ transforms the linearized operator of the $b$-family into a form closely resembling that of the gKdV equation. This allows us to exploit the well-understood spectral properties of the localized gKdV-type operator to fully characterize the spectrum of the $b$-family linearization.
A complication arises, however: under this transformation the eigenvalue problem becomes
$\mathcal{L}f=\lambda\sqrt{\alpha}f$,
where $\alpha$ (defined in \eqref{wei}) is uniformly bounded away from zero. As a result, eigenfunctions associated with negative eigenvalues cannot be directly employed in the spectral analysis. It is therefore necessary to identify suitable test functions that preserve localization under the action of $\alpha$. In Section~\ref{s2}, we introduce the specific function
$$SQ=\frac{2k(1-b)}{b}Q^{-1/b}+\frac{2(b-1)}{b}Q,
$$
which satisfies these requirements and plays a central role in our analysis.

Although special functions arise in this context, it remains essential to verify the following nonstandard spectral condition:
\begin{equation}\label{sc}
\int \mathcal{L}^{-1} (SQ), (SQ), \mathrm{d}x < 0.
\end{equation}
The coercivity of the associated bilinear form holds provided that the specific function $SQ$ satisfies both the spectral condition \eqref{sc} and the orthogonality condition
$$\int vSQ\,\mathrm{d}x=0,$$
while at the same time not being orthogonal to the eigenfunction of $\mathcal{L}$. Importantly, the validity of \eqref{sc} depends on the parameter regime $b<-1$, which corresponds to $k>1$. Outside this regime the spectral condition fails, and coercivity cannot be established.
A further difficulty arises from the structure of the conservation law $F_2$, which is not naturally associated with the $H^1$ energy space. This lack of direct correspondence obstructs the derivation of the $H^1$-based conservation laws needed for virial estimates. In particular, this necessitates the additional assumption $F_2<\infty$ in the proof of global well-posedness \cite{HL}. Since $F_2$ cannot be directly exploited, we instead impose a positivity assumption on $m$, which ensures boundedness of solutions in $H^1_\alpha$. This positivity condition is likewise crucial for establishing global well-posedness.

The remainder of the paper is organized as follows. The monotonicity results and linear Liouville property (Theorem \ref{llt}) are established in Section \ref{s2}. Section \ref{s3} is devoted to the crucial monotonicity arguments for $m$ and the nonlinear Liouville property (Theorem \ref{nlt}), required for the proof of Theorem \ref{as}. The proof of asymptotic stability (Theorem \ref{as}) is given in Section \ref{s4}, focusing on the verification of convergence results and compactness conditions.

\section{Linear Liouville property}\label{s2}
The proof of Theorem \ref{llt}, developed in this section, is inspired by  the methodology outlined in \cite{MM1}. We first introduce a smooth, monotonic weight function   and define three localized integral quantities to derive the monotonicity estimates. We then establish uniform exponential decay for all spatial derivatives of linearized solutions of \eqref{l2c} (Lemma \ref{hol}). To prove the theorem, we define a transformed variable  $\eta$ \eqref{dp}  (satisfying orthogonality conditions) and construct a virial identity \eqref{w1} with weight. By employing the coercivity property satisfied by the quadratic form \eqref{lzze} (its associated linear operator linked to gKdV operators via variable transformation) and \eqref{wlz}, we show that localized linear solutions of \eqref{l2c} must be of the form \eqref{lltt}.
\subsection{Monotonicity and smoothness}
To establish the linear Liouville theorem, the main component of proof is the monotonicity properties for the solutions of \eqref{l2c}.
The proof of the monotonicity property relies on an infinitely differentiable weight functions, with the key properties being its monotonic and bounded nature, and the locally convergent behavior exhibited by its derivative.
We define a function $\psi_L(x)\in C^{\infty}(\mathbb{R})$ by
\begin{equation}\label{tf}
\psi_L(x):=\frac{2}{\pi}\arctan\left(\exp\left(\frac{\nu x}{L}\right)\right),
\end{equation}
where $\nu=-\frac{b+1}{2}$ and the positive constant $L=-b+3>4$ with $b<-1$. A straightforward computation shows that
$$\lim_{x\to+\infty}\psi_L(x)=1,~\lim_{x\to-\infty}\psi_L(x)=0,$$ and
$$\psi_L'(x)=\frac{\nu}{\pi L}\operatorname{sech}\left(\frac{\nu x}{L}\right),$$
which satisfies
\begin{equation}\label{psi}
  \psi_L'(x)\leq\psi_L(x),~~\left|\psi_L'''(x)\right|\leq\frac{-b}{L^2}\psi_L'(x),~~~~\forall x\in\mathbb{R}.
\end{equation}

For any $x_0>0$ and $t_0\in\mathbb{R}$, we introduce the following two quantities
$$\mathcal{I}(t)=\int v^2(x,t)\psi_L(\tilde{x})\,\mathrm{d}x,~~~~\mathcal{J}(t)=\int v_x^2(x,t)\psi_L(\tilde{x})\,\mathrm{d}x,$$
where $v$ is the solutions of \eqref{l2c}, $\alpha$ is defined in \eqref{wei} and $\tilde{x}=x+8b(t-t_0)-x_0$ for $t_0\leq t.$
We have the following monotonicity properties for $v(x,t)$. Prior to establishing the monotonicity, we compute $\mathcal{B}(Q)\mathcal{L},$ the detailed derivation provided in Appendix.
\begin{lemma}\label{lmp1}
  Let $v\in C\left(\mathbb{R},H_{\alpha}^1\right)\cap L^{\infty}\left(\mathbb{R},H_{\alpha}^1\right)$ be the solutions of \eqref{l2c} satisfying \eqref{lceta}. For any $x_0>0$, $t_0\in\mathbb{R}$ and $t_0\leq t_1,$ it holds that
  \begin{equation}\label{eta1}
   \mathcal{I}(t_1)-\mathcal{I}(t_0)\lesssim e^{-\nu x_0/L},~~\mathcal{J}(t_1)-\mathcal{J}(t_0)\lesssim e^{-\nu x_0/L}.
  \end{equation}
\end{lemma}
\begin{proof}
  Applying integration by parts to \eqref{l2c} and utilizing the properties of $\psi_L(x)$ together with \eqref{ble}, it follows that
  \begin{align}\label{ve1}
  \frac12\frac{\mathrm{d}}{\mathrm{d}t}&\int \left(v^2+v_x^2\right)\psi_L(\tilde{x})\,\mathrm{d}x\nonumber\\
  =&\int \left(v\partial_t v+v_x\partial_t v_x\right)\psi_L(\tilde{x})\,\mathrm{d}x
    +4b\int\left(v^2+v_x^2\right)\psi'_L(\tilde{x})\,\mathrm{d}x\nonumber\\
    =&\int v\left(\frac{2k}{1-b}Q^{-\frac{1}{b}-1}Q'v-\frac{2k}{1-b}Q^{-\frac{1}{b}}v_x\right)\psi_L(\tilde{x})\,\mathrm{d}x+\int v\left(-bQh_x-Q'h\right)\psi_L(\tilde{x})\,\mathrm{d}x\nonumber\\
    &+\int v_x\left(-\frac{2kb}{1-b}Q^{-\frac{1}{b}}v+bQv+\frac{2k(1+b)}{b(1-b)}Q^{-\frac{1}{b}-1}Q'v_x-\frac{2k}{1-b}Q^{-\frac{1}{b}}v_{xx}\right)\psi_L(\tilde{x})\,\mathrm{d}x\nonumber\\
    &+\int v_x\left(-bQ'h_x-bQh_{xx}-Q''h-Q'h_x\right)\psi_L(\tilde{x})\,\mathrm{d}x+4b\int\left(v^2+v_x^2\right)\psi'_L(\tilde{x})\,\mathrm{d}x\nonumber\\
    =&4b\int\left(v^2+v_x^2\right)\psi'_L(\tilde{x})\,\mathrm{d}x+\frac{k(2b+1)}{b(1-b)}\int Q^{-\frac{1}{b}-1}Q'v_x^2\psi_L(\tilde{x})\,\mathrm{d}x+\frac{k(b+1)}{1-b}\int Q^{-\frac{1}{b}}v^2\psi_L'(\tilde{x})\,\mathrm{d}x\nonumber\\
    &-\frac{k}{b}\int Q^{-\frac{1}{b}-1}Q' v^2\psi_L(\tilde{x})\,\mathrm{d}x+\frac{k}{1-b}\int Q^{-\frac{1}{b}}v_x^2\psi_L'(\tilde{x})\,\mathrm{d}x+(b+1)\int Q'v h_x\psi_L'(\tilde{x})\,\mathrm{d}x\nonumber\\
    &+2b\int Q' v h\psi_L(\tilde{x})\,\mathrm{d}x+(b+2)\int Q''vh_x\psi_L(\tilde{x})\,\mathrm{d}x+b\int Q v h\psi_L'(\tilde{x})\,\mathrm{d}x-b\int Q v^2\psi_L'(\tilde{x})\,\mathrm{d}x\nonumber\\
    &-(2b+1)\int Q' v^2 \psi_L(\tilde{x})\,\mathrm{d}x+\int Q'' vh \psi_L'(\tilde{x})\,\mathrm{d}x+\int Q''' vh \psi_L(\tilde{x})\,\mathrm{d}x.
  \end{align}
where $h:=(1-\partial_x^2)^{-1}v.$ We deduce from the Cauchy-Schwarz inequality that
\begin{align}\label{vve}
  \frac12\frac{\mathrm{d}}{\mathrm{d}t}\int \left(v^2+v_x^2\right)\psi_L(\tilde{x})\alpha\,\mathrm{d}x
    \leq&4b\int\left(v^2+v_x^2\right)\psi'_L(\tilde{x})\,\mathrm{d}x\nonumber\\
    +&\frac{k(2b+1)}{b(1-b)}\int Q^{-\frac{1}{b}-1}Q'v_x^2\psi_L(\tilde{x})\,\mathrm{d}x+\frac{k(b+1)}{1-b}\int Q^{-\frac{1}{b}}v^2\psi_L'(\tilde{x})\,\mathrm{d}x\nonumber\\
    -&\frac{k}{b}\int Q^{-\frac{1}{b}-1}Q' v^2\psi_L(\tilde{x})\,\mathrm{d}x+\frac{k}{1-b}\int Q^{-\frac{1}{b}}v_x^2\psi_L'(\tilde{x})\,\mathrm{d}x\nonumber\\
    +&\int v^2\bigg(-3bQ\psi'_L(\tilde{x})-(3b+1)Q'\psi_L(\tilde{x})-\frac{b+1}{2}Q'\psi'_L(\tilde{x})\nonumber\\
    &+\frac{2-b}{2}Q''\psi_L(\tilde{x})+Q''\psi'_L(\tilde{x})+Q'''\psi_L(\tilde{x})\bigg)\,\mathrm{d}x\nonumber\\
    +&\int h^2\bigg(-\frac{b}{2}Q\psi'_L(\tilde{x})-b Q'\psi_L(\tilde{x})+Q''\psi'_L(\tilde{x})+Q'''\psi_L(\tilde{x})\bigg)\,\mathrm{d}x\nonumber\\
    +&\int h_x^2\bigg(-\frac{b+1}{2}Q'\psi'_L(\tilde{x})+\frac{2-b}{2} Q'''\psi_L(\tilde{x})\bigg)\,\mathrm{d}x.
  \end{align}
We begin with the following two facts that
$$Q''=b^2Q-\frac{b(3b+1)}{2k}Q^{\frac1b+2}$$
and
$$\left|b^2-\frac{b(3b+1)}{2k}Q\right|+\left|b^2-\frac{(2b+1)(3b+1)}{2k}Q^{\frac1b+1}\right|\lesssim_{b}1,$$
which together with the properties of lefton $Q$ yield that
\begin{align}\label{efq}
    \Bigg|\frac{k(2b+1)}{b(1-b)}&Q^{-\frac{1}{b}-1}Q'\psi_L(\tilde{x})+\frac{k}{-b}Q^{-\frac{1}{b}-1}Q'\psi_L(\tilde{x})+Q\psi'_L(\tilde{x})\nonumber\\
    &+Q'(\psi_L(\tilde{x})+\psi'_L(\tilde{x}))+Q''(\psi_L(\tilde{x})+\psi'_L(\tilde{x}))+Q'''\psi_L(\tilde{x})\Bigg|\leq(-b)\operatorname{sech}\left(\frac{\nu x}{L}\right)\psi_L(\tilde{x}),
\end{align}
for $b<-1$ and $L>4.$
It remains to estimate the right hand of \eqref{vve}. We consider three cases about $x$ for some constant $R_0>0.$\\
$Case~1.$ If $x<R_0,$ it indicates that $\tilde{x}<R_0+8b(t-t_0)-x_0,$ then we have
  $$\operatorname{sech}\left(\frac{\nu x}{L}\right)\psi_L(\tilde{x})\lesssim e^{\nu\tilde{x}/L}\lesssim e^{\nu\left(R_0+8b(t-t_0)-x_0\right)/L}.$$
$Case~2.$ If $R_0<x<4b(t_0-t)+\frac12x_0,$ then $\tilde{x}<0.$ A direct computation gives  that
  $$\operatorname{sech}\left(\frac{\nu x}{L}\right)\psi_L(\tilde{x})\lesssim \operatorname{sech}\left(\frac{\nu x}{L}\right)\psi_L'(\tilde{x})$$
  since  $\psi_L(\tilde{x})\lesssim\psi_L'(\tilde{x})$ for $\tilde{x}<0.$\\
$Case~3.$ If $x>4b(t_0-t)+\frac12x_0>0,$ we have
  $$\operatorname{sech}\left(\frac{\nu x}{L}\right)\psi_L(\tilde{x})\lesssim \operatorname{sech}\left(\frac{\nu x}{L}\right)\lesssim \operatorname{sech}\left(\frac{\nu}{L} \left(4b(t_0-t)+\frac12x_0\right)\right)\lesssim e^{-\nu\left|4b(t_0-t)+\frac12x_0\right|/L},$$
since $\psi_L'(x)\lesssim e^{-|x|/L}$ for $x\in\mathbb{R}.$

It follows from \eqref{psi} that
$$\int (h^2+h_x^2)\psi_L'(\tilde{x})\,\mathrm{d}x\lesssim\int v^2\psi_L'(\tilde{x})\,\mathrm{d}x,$$
which gives that
\begin{align}\label{vev}
 \frac12&\frac{\mathrm{d}}{\mathrm{d}t}\int \left(v^2+v_x^2\right)\psi_L(\tilde{x})\,\mathrm{d}x\nonumber\\
 \leq&3b\int\left(v^2+v_x^2\right)\psi'_L(\tilde{x})\,\mathrm{d}x+Ce^{\nu\left(R_0+8b(t-t_0)-x_0\right)/L}\int \left(v^2+v_x^2+h^2+h_x^2\right)\,\mathrm{d}x\nonumber\\
 &+Ce^{-\nu\left|4b(t_0-t)+\frac12x_0\right|/L}\int \left(v^2+v_x^2+h^2+h_x^2\right)\,\mathrm{d}x\nonumber\\
 &+(-b)\operatorname{sech}\left(\frac{\nu R_0}{L}\right)\int \left(v^2+v_x^2+h^2+h_x^2\right)\psi_L'(\tilde{x})\,\mathrm{d}x\nonumber\\
 \leq&2b\int\left(v^2+v_x^2\right)\psi'_L(\tilde{x})\,\mathrm{d}x+Ce^{\nu\left(R_0+8b(t-t_0)-x_0\right)/L}\int \left(v^2+v_x^2\right)\,\mathrm{d}x\nonumber\\
 &+Ce^{-\nu\left|4b(t_0-t)+\frac12x_0\right|/L}\int \left(v^2+v_x^2\right)\,\mathrm{d}x,
\end{align}
for sufficiently large $R_0$ such that $\operatorname{sech}\left(\frac{\nu R_0}{L}\right)\leq\frac{1}{2}.$

Integrating this inequality on $(t_0-1,t_0)$, we obtain that, for a positive  constant $C>0,$
\begin{equation*}
    \mathcal{I}(t_0)-\mathcal{I}(t_0-1)\leq Ce^{-\nu x_0/L}~~,\forall t_0\in\mathbb{R}.
  \end{equation*}
Similarly, integrating \eqref{vev} on $[t_0,t_0+\tau]$ for any $\tau>0,$ we have
\begin{equation*}
    \mathcal{I}(t_0+\tau)-\mathcal{I}(t_0)\leq Ce^{-\nu x_0/L}e^{-\nu(t_0+\tau)/L}.
  \end{equation*}
It is clear that
$$\mathcal{I}(t_0-1)=\int v^2(x,t_0-1)\psi_L(\tilde{x}(t_0-1))\,\mathrm{d}x\lesssim 1$$
and taking the limit $\tau\to\infty,$ we deduce that
$$\mathcal{I}(t_0+\tau)=\int v^2(x,t_0+\tau)\psi_L(x+8b\tau-x_0)\,\mathrm{d}x\to0$$
and $\mathcal{I}(t_0)\geq-Ce^{-\nu x_0/L},$ which implies that $|\mathcal{I}(t_0)|\lesssim e^{-x_0/L}.$ For $\mathcal{J}(t_0)$, it follows similarly that $|\mathcal{J}(t_0)|\lesssim e^{-\nu x_0/L}.$
In addition, integrating \eqref{vev} on $(\tilde{t}_0,\tilde{t}_0+1)$, we get that
\begin{align}
\int&\left(v^2+v_x^2\right)(x,t_0)\psi_L(x-x_0)\,\mathrm{d}x+\int_{\tilde{t}_0}^{\tilde{t}_0+1}\int\left(v^2+v_x^2\right)\psi'_L(\tilde{x})\,\mathrm{d}x\,\mathrm{d}t
    \lesssim e^{-\nu x_0/L},\label{c2.8}
\end{align}
for $ \tilde{t}_0\in\mathbb{R}.$
\end{proof}

We continue the proof of the monotonicity properties for solutions with high-order derivatives,  the counterpart gKdV result was shown in \cite{MM8}.

\begin{lemma}\label{lmp2}
  Let $v\in C\left(\mathbb{R},H_{\alpha}^1\right)\cap L^{\infty}\left(\mathbb{R},H_{\alpha}^1\right)$ be the solutions of \eqref{l2c} satisfying \eqref{lceta}. For any $x_0>0$, $t_0\in\mathbb{R}$ and $t_0\leq t_1,$ it holds that
  \begin{equation}\label{eta11}
    \mathcal{K}(t_1)-\mathcal{K}(t_0)\lesssim e^{-\nu x_0/L},
  \end{equation}
where $\mathcal{K}(t)=\int (\partial_x^j v)^2\psi_L(\tilde{x})\,\mathrm{d}x$ for $j\in\mathbb{N}.$
\end{lemma}

\begin{proof}
We deduce from \eqref{ble} and the Leibniz rule that
\begin{align*}
    \partial_t\left(\partial_x^j v\right)=&\frac{2k}{1-b}\sum_{0\leq i\leq j}\partial_x^i\left(-b\partial_x\left(Q^{-\frac1b}\right)\right)\partial_x^{j-i}v
    -\frac{2k}{1-b}\sum_{0\leq i\leq j}\partial_x^i\left(Q^{-\frac1b}\right)\partial_x^{j+1-i}v\nonumber\\
    &-b\sum_{0\leq i\leq j}\partial_x^iQ\partial_x^{j+1-i}h-\sum_{0\leq i\leq j}\partial_x^{i+1}Q\partial_x^{j-i}h
\end{align*}
Arguing as Lemma \ref{lmp1}, for $v,$ we obtain that
  \begin{align}
    \frac12&\frac{\mathrm{d}}{\mathrm{d}t}\int\left(\partial_x^jv\right)^2\psi_L(\tilde{x})\,\mathrm{d}x\nonumber\\
    =&\frac{2k}{1-b}\sum_{0< i\leq j}\int\partial_x^{i+1}\left(Q^{-\frac1b}\right)\partial_x^{j-i}v \partial_x^{j}v\psi_L(\tilde{x})\,\mathrm{d}x
    -\frac{2k}{1-b}\sum_{1< i\leq j}\int\partial_x^{i}\left(Q^{-\frac1b}\right)\partial_x^{j-i}v \partial_x^{j}v\psi_L(\tilde{x})\,\mathrm{d}x\nonumber\\
    &-b\sum_{0< i\leq j}\int\partial_x^{i}Q\partial_x^{j+1-i}h\partial_x^{j}v\psi_L(\tilde{x})\,\mathrm{d}x-\sum_{0< i\leq j}\int\partial_x^{i+1}Q\partial_x^{j-i}h \partial_x^{j}v\psi_L(\tilde{x})\,\mathrm{d}x\nonumber\\
    &\frac{k(2b+1)}{b(1-b)}\int Q^{-\frac{1}{b}-1}Q'\left(\partial_x^jv\right)^2\psi_L(\tilde{x})\,\mathrm{d}x+\frac{k}{1-b}\int Q^{-\frac{1}{b}}\left(\partial_x^jv\right)^2\psi_L'(\tilde{x})\,\mathrm{d}x\nonumber\\
    &-b\int Q\partial_x^{j+1}h\partial_x^jv\psi_L(\tilde{x})\,\mathrm{d}x-\int Q'\partial_x^{j}h\partial_x^jv\psi_L(\tilde{x})\,\mathrm{d}x.
  \end{align}
We claim that
  \begin{align}\label{efho}
    \frac12\frac{\mathrm{d}}{\mathrm{d}t}\int\left(\partial_x^jv\right)^2\psi_L(\tilde{x})\,\mathrm{d}x\leq& -3b\int\left(\partial_x^jv\right)^2\psi_L'(\tilde{x})\,\mathrm{d}x-\frac{b}{4}\int\left((\partial_x^{j+1}h)^2+(\partial_x^{j}h)^2\right)^2\psi_L'(\tilde{x})\,\mathrm{d}x\nonumber\\
    &-\frac{b}{2}\sum_{0\leq i\leq j-2}\int\left(\partial_x^{i}v\right)^2\psi_L'(\tilde{x})\,\mathrm{d}x-\frac{b}{8}\sum_{0\leq i\leq j-1}\int\left(\partial_x^{i+1}h\right)^2\psi_L'(\tilde{x})\,\mathrm{d}x\nonumber\\
    &-\frac{b}{2}\sum_{0\leq i\leq j-1}\int\left(\partial_x^{i}v\right)^2\psi_L'(\tilde{x})\,\mathrm{d}x-\frac{b}{8}\sum_{0\leq i\leq j-1}\int\left(\partial_x^{i}h\right)^2\psi_L'(\tilde{x})\,\mathrm{d}x\nonumber\\
    \leq&-3b\int\left(\partial_x^jv\right)^2\psi_L'(\tilde{x})\,\mathrm{d}x-\frac{3b}{2}\sum_{0\leq i\leq j-1}\int\left(\partial_x^i v\right)^2\psi_L'(\tilde{x})\,\mathrm{d}x.
  \end{align}
  Indeed, it is observed from the properties of $Q$, the fact $\frac{k}{1-b}<-\frac{b}{2},~\frac{k(2b+1)}{b(1-b)}<-\frac{b}{2},-\frac{k}{b}<-b$ and
  $$\partial_x^2\left(Q^{-\frac{1}{b}-1}Q'\right)=-bQ^{-\frac{1}{b}}+\frac{b(1-b)}{2k}Q\lesssim\operatorname{sech}(\nu x/L),$$
  that
\begin{align}\label{psi2}
    \operatorname{sech}(\nu x/L)\psi_L(\tilde{x})\lesssim \psi_L'(\tilde{x}).
\end{align}
Here, we consider two cases for $\tilde{x}.$ In the case $\tilde{x}<0,$ it is clear that $\psi_L(\tilde{x})\lesssim \psi_L'(\tilde{x}).$ For the case $\tilde{x}>0,$ it reveals that $x>-8b(t-t_0)+x_0,$ that is $x>\tilde{x}>0,$ which leads to
$$\operatorname{sech}(\nu x/L)\psi_L(\tilde{x})\lesssim\operatorname{sech}(\nu x/L)\lesssim\psi_L'(\tilde{x}).$$
Invoking \eqref{psi2} and the properties of $Q$, we get that \eqref{efho} holds for $b<-1.$
In a similar way, we obtain that
\begin{align}\label{efho2}
    \frac12\frac{\mathrm{d}}{\mathrm{d}t}\int\left(\partial_x^{j+1}v\right)^2\psi_L(\tilde{x})\,\mathrm{d}x
    \leq& b\int\left(\left(\partial_x^jv\right)^2+\left(\partial_x^{j+1}v\right)^2\right)\psi_L'(\tilde{x})\,\mathrm{d}x\nonumber\\
    &-\frac{3b}{2}\sum_{0\leq i\leq j-1}\int\left(\partial_x^i v\right)^2\psi_L'(\tilde{x})\,\mathrm{d}x.
  \end{align}

We integrate \eqref{efho2} between $t_0-1$ and $t_0$, for fixed $t_0\in\mathbb{R},$ to obtain that
\begin{equation}\label{ho2}
  \mathcal{K}(t_0)-\mathcal{K}(t_0-1)\lesssim e^{-\nu x_0 / L}.
\end{equation}
For $\mathcal{K}(t_0-1),$ invoking the induction hypothesis for $j-1,$ we have
$$\int_{\tilde{t}_0}^{\tilde{t}_0+1}\int\left(\partial^{j}_x v\right)^2(x, t) \psi_L^{\prime}\left(\tilde{x}\right) \mathrm{d} x \mathrm{d} t \lesssim e^{-\nu x_0 / L},$$
which leads to
$$\int_{\tilde{t}_0}^{\tilde{t}_0+1}\int_{x<8b(t_0-t)+x_0}\left(\partial^{j}_x v\right)^2(x, t) e^{\nu\left(x+8b(t_0-t)\right)} \mathrm{d} x \mathrm{d} t \lesssim 1.$$
Then, we deduce from passing to the limit as $x_0\to\infty$ and multiplying $e^{-\nu x_0 / L}$ that
$$\int_{\tilde{t}_0}^{\tilde{t}_0+1}\int\left(\partial^{j}_x v\right)^2(x, t) \psi_L\left(\tilde{x}\right) \mathrm{d} x \mathrm{d} t \lesssim e^{-\nu x_0 / L},$$
where we use the fact that $\psi'_L(x)\geq e^{\nu x/L}$ for $x<0$ and $\psi_L(x)\leq e^{\nu x/L}$ for $x\in\mathbb{R}.$ In fact, we take a special choose $\tilde{t}_0=t_0-\frac32$ with $t_0-1\in(\tilde{t}_0,\tilde{t}_0+1)$ to get that $\mathcal{K}(t_0-1)\lesssim e^{-\nu x_0 / L}.$ We finally integrate \eqref{efho2} between $t_0$ and $t_0+\tau$ for $\tau>0$ to yield that
$$\mathcal{K}(t_0+\tau)-\mathcal{K}(t_0)\lesssim e^{-\nu x_0 / L}.$$
Taking the limit $\tau\to\infty,$ we also deduce that
$\mathcal{K}(t+\tau)\to0$ and $\mathcal{K}(t_0)\geq-Ce^{-\nu x_0/L},$ which imply that $|\mathcal{K}(t_0)|\lesssim e^{-\nu x_0/L}.$ This completes the proof of Lemma \ref{lmp2}.
\end{proof}

In the following, we derive the uniform exponential decays of the solutions  of \eqref{l2c}.

\begin{lemma}\label{hol}
  Let $v\in C\left(\mathbb{R},H_{\alpha}^1\right)\cap L^{\infty}\left(\mathbb{R},H_{\alpha}^1\right)$ be the solutions of \eqref{l2c} satisfying \eqref{lceta}. Then there exists a positive constant $\nu$ such that for any $j\in\mathbb{N}$,  there holds
  \begin{equation}\label{eta12}
    \sup_{t\in\mathbb{R}}\int\left(\partial_x^j v\right)^2(x,t)e^{\nu|x|}\,\mathrm{d}x\lesssim1.
  \end{equation}
\end{lemma}
\begin{proof}
Under the conclusions of Lemma \ref{lmp2}, we derive from $\phi_L(y)\geq Ce^{y/L}$ and $\phi_L'(y)\geq Ce^{y/L}$ for $y<0$ that
  \begin{align}\label{hod}
   \int_{x<x_0}&\left(\partial_x^jv\right)^2(x,t_0)e^{\nu x/L}\,\mathrm{d}x\lesssim 1.
  \end{align}
Observe that $v(-x,-t)$ is the solutions of \eqref{l2c} satisfying \eqref{lceta}. Consequently, the same result holds for $v(-x,-t)$ when choosing $\delta=\frac{\nu}{L}$, which implies that \eqref{eta12} holds. Moreover, one gives that
\begin{align}\label{hod1}
  e^{A_0\delta}\int_{x>A_0}&\left(\partial_x^j v\right)^2(x,t)\,\mathrm{d}x\lesssim \int\left(\partial_x^jv\right)^2(x,t)e^{\delta|x|}\,\mathrm{d}x,~~A_0>0,
\end{align}
and invoking the Gagliardo-Nirenberg inequality $\|v\|_{L^{\infty}(x>A_0)}^2\lesssim \|v\|_{L^{2}(x>A_0)}\|\partial_x v\|_{L^{2}(x>A_0)},$ we obtain the pointwise exponential decay
$|\partial_x^j v|\lesssim e^{-A_0\nu}.$ We argue similarly for the derivative in space at any order of $v(-x,-t),$ that is $|\partial_x^j \tilde{v}|\lesssim e^{-A_0\nu}$ for $\tilde{v}=v(-x,-t),$ which completes the proof of Lemma \ref{hol}.
\end{proof}

\subsection{Proof of linear Liouville theorem}\label{s2.2}
In this subsection, we give the proof of Theorem \ref{llt}. The idea of proof originates from the strategy in \cite{MM1}, which introduces a dual problem that simplifies
\eqref{l2c}. We define that
\begin{equation}\label{dp}
    \eta=\mathcal{L}v-\frac{b+1}{b}\beta(t),
\end{equation}
where $$\beta(t)=\frac{b}{b+1}\frac{\int SQ\mathcal{L}v\,\mathrm{d}x}{\int SQ\,\mathrm{d}x},~~SQ=\frac{2k(1-b)}{b}Q^{-\frac{1}{b}}+\frac{2(b-1)}{b}Q.$$
It is readily verifiable $\int SQ\,\mathrm{d}x\neq0,$ which also ensures that $\beta(t)$ is well-defined. Note also that $\eta$ satisfies the following orthogonality conditions
\begin{equation}\label{woc1}
\int \eta Q'\,\mathrm{d}x=\int \mathcal{L}v Q'\,\mathrm{d}x-\frac{b+1}{b}\beta(t)\int Q'\,\mathrm{d}x=\int v \mathcal{L}Q'\,\mathrm{d}x-\frac{b+1}{b}\beta(t)\int Q'\,\mathrm{d}x=0,
\end{equation}
and
\begin{equation}\label{woc2}
\int \eta SQ\,\mathrm{d}x=\int \mathcal{L}v SQ\,\mathrm{d}x-\frac{b+1}{b}\beta(t)\int SQ\,\mathrm{d}x=0.
\end{equation}
We derive from \eqref{ker}, \eqref{l2c} and \eqref{dp} that
\begin{align}\label{dp2}
    \partial_t\eta=\mathcal{L}\partial_t v-\frac{b+1}{b}\beta'(t)=&\mathcal{L}\left(\frac{1}{1-b}\mathcal{B}(Q)\mathcal{L}v\right)+\tilde{\beta}(t)\nonumber\\
    =&\frac{1}{1-b}\mathcal{L}\mathcal{B}(Q)\eta+\tilde{\beta}(t),
\end{align}
where $$\tilde{\beta}(t)(t)=-\frac{\frac{\mathrm{d}}{\mathrm{d}t}\int SQ\mathcal{L}v\,\mathrm{d}x}{\int SQ\,\mathrm{d}x},$$
and we used the fact $\mathcal{L}\mathcal{B}(Q)\left(c\right)=0$ for arbitrary constant $c.$

We note that there is a compact condition for $\eta\in C\left(\mathbb{R},H_{\alpha}^1\right)\cap L^{\infty}\left(\mathbb{R},H_{\alpha}^1\right).$
For a positive constant $C,$ it follows from Lemma \ref{hol} and the decay properties of $Q$ that
\begin{equation}\label{hode}
  \left|\eta(x,t)\right|\leq Ce^{-\nu|x|},\quad\forall x\in\mathbb{R},t\in\mathbb{R}.
\end{equation}

Next, in order to compute a virial identity concerning the solutions of \eqref{l2c}, we introduce two functions $\phi(x)=Q'(x)$ and $\Phi(x)=b^2Q(x).$ For a constant $M>1$ to be chosen later, we set
\begin{equation}\label{phim}
\Phi_M(x)=b^2Q\left(\frac{x}{M+t^2}\right).
\end{equation}
A direct calculation gives
\begin{equation}\label{vgf}
  \phi'(x)=Q''(x)=b^2Q-\frac{b(3b+1)}{2k}Q^{\frac1b+2},~~\phi''(x)=b^2Q'-\frac{(2b+1)(3b+1)}{2k}Q^{\frac1b+1}Q',
\end{equation}
and
\begin{equation}\label{vgf1}
  \phi'''(x)=-\frac{(2b+1)(3b+1)}{2k}\left((2b^2+b)Q^{\frac1b+2}-\frac{b(5b+3)}{2k}Q^{\frac2b+3}\right)-\frac{b^3(3b+1)}{2k}Q^{\frac1b+2}+b^4Q.
\end{equation}
Similarly, we also have
\begin{equation}\label{vgf2}
  \frac{(\Phi')^2}{\Phi'}=b^2\frac{(Q')^2}{Q}=b^4Q\left(1-\frac{Q^{\frac{1}{b}+1}}{k}\right).
\end{equation}
We compute that there exists a positive constant $C$ with respect to $b$ such that
\begin{equation}\label{gine}
\left|\phi'(x)\right|+\left|\frac{(\phi'')^2}{\phi'}\right|+\left|\phi'''(x)\right|\leq C\Phi(x),~~\left|\Phi'\right|+\left|\frac{(\Phi')^2}{\Phi'}\right|\leq C\Phi(x).
\end{equation}
Multiplying the equation of $\eta$ in \eqref{dp2} and integrating over $\mathbb{R}$, we infer from the expression of $\eta\mathcal{L}\mathcal{B}(Q)$ \eqref{lbe} that
\begin{align}\label{w1}
\frac12\frac{\mathrm{d}}{\mathrm{d}t}\int \eta^2\phi\,\mathrm{d}x=&\frac{1}{1-b}\int \eta\mathcal{L}\mathcal{B}(Q)\eta\phi\,\mathrm{d}x+\tilde{\beta}\int \eta\phi\,\mathrm{d}x\nonumber\\
=&\frac{1}{1-b}\int \eta\phi\bigg(-2kQ^{-\frac{1}{b}}\partial_x\eta+\frac{2k(1-b)}{b}Q^{-\frac{1}{b}-1}Q'\eta\nonumber\\
    &+(b-1)(1-\partial_x^2)^{-1}\left(bQ\partial_x\eta+(b-1)Q'\eta\right)\bigg)\,\mathrm{d}x\nonumber\\
=&-\frac{k(2b-1)}{b(1-b)}\int Q^{-\frac{1}{b}-1}Q'\eta^2\phi\,\mathrm{d}x+\frac{k}{(1-b)}\int Q^{-\frac{1}{b}}\eta^2\phi'\,\mathrm{d}x\nonumber\\
&-\int \eta\phi(1-\partial_x^2)^{-1}\left(bQ\partial_x\eta+(b-1)Q'\eta\right)\,\mathrm{d}x,
\end{align}
where the $\tilde{\beta}$ term vanish by employing the orthogonality conditions \eqref{woc1}. Unlike the Hamiltonian structure of the KdV equation, the virial identity cannot directly yield the expression for operator $\mathcal{L}$ due to the presence of linear nonlocal operator $\mathcal{B}$. Therefore, we directly compute the virial estimate for operator $\mathcal{L}$ herein.

Let $M>1$ be chosen later and set $z=\eta\sqrt{\Phi_M}$ such that
\begin{align}\label{lz}
    \left(\mathcal{L}z,z\right)=&\frac{2k}{b^2}\int\eta_x^2\Phi_M\alpha\,\mathrm{d}x-\frac{2k(b+1)}{b}\int\eta^2\Phi_M\alpha\,\mathrm{d}x-\frac{k}{b^2}\int\eta^2\Phi_M'\alpha_x\,\mathrm{d}x\nonumber\\
    &-\frac{k}{b^2}\int\eta^2\Phi_M''\alpha\,\mathrm{d}x+\frac{k}{2b^2}\int\eta^2\frac{(\Phi_M')^2}{\Phi_M}\alpha\,\mathrm{d}x.
\end{align}
For $b<-1,$ we claim that the following {\bf coercivity property} holds true,
\begin{align}\label{lzze}
\left(\mathcal{L}z,z\right)\geq\lambda\int(\eta^2\alpha+\eta_x^2\alpha)\Phi_M\,\mathrm{d}x.
\end{align}
where $\lambda>0$ will be determined subsequently. To prove this coercivity property, we first present the following lemma, which concerns the first two terms at the right hand of \eqref{w1}.
\begin{lemma}\label{bfg1}
  For the bilinear form
  \begin{equation}\label{bf1}
   L(\eta,\eta):=\frac{2k}{b^2}\int\eta_x^2\Phi_M\alpha\,\mathrm{d}x-\frac{2k(b+1)}{b}\int\eta^2\Phi_M\alpha\,\mathrm{d}x,
  \end{equation}
  there exists a positive constant $\kappa$ such that
  \begin{equation}\label{bfw1}
    L(\eta,\eta)\geq\kappa\int(\eta^2\alpha+\eta_x^2\alpha)\Phi_M\,\mathrm{d}x,
  \end{equation}
where $\eta\in C\left(\mathbb{R},H_{\alpha}^1\right)\cap L^{\infty}\left(\mathbb{R},H_{\alpha}^1\right)$ satisfying the orthogonality conditions \eqref{woc1} and \eqref{woc2}.
\end{lemma}

To prove Lemma \ref{bfg1}, we consider a non-localized bilinear form
\begin{equation}\label{bfwi1}
(\mathcal{L}\eta,\eta)=\frac{2k}{b^2}\int\eta_x^2\alpha\,\mathrm{d}x-\frac{2k(b+1)}{b}\int\eta^2\alpha\,\mathrm{d}x.
  \end{equation}
The proof of Lemma \ref{bfg1} involves the following coercivity property of $(\mathcal{L}\eta,\eta),$ which is similar to Lemma E.1 in \cite{MW}.
\begin{lemma}\label{bfg11}
There exists a positive constant $\lambda_1>0$ such that
 \begin{equation}\label{bf11}
    (\mathcal{L}\eta,\eta)\geq\lambda_1\left\|\eta\right\|_{H^1_{\alpha}}^2,
  \end{equation}
for any $\eta\in C\left(\mathbb{R},H_{\alpha}^1\right)\cap L^{\infty}\left(\mathbb{R},H_{\alpha}^1\right)$ satisfying the orthogonality conditions \eqref{woc1} and \eqref{woc2}.
\end{lemma}
\begin{proof}
It is found  that
\begin{equation}\label{e1c1}
  (SQ,\chi_0)_{\alpha}=\frac{1}{\lambda_0}(SQ,\mathcal{L}\chi_0)=\frac{1}{\lambda_0}(\mathcal{L}(SQ),\chi_0)\neq0,
\end{equation}
where
\begin{align*}
  \mathcal{L}(SQ)=&\mathcal{L}\left(\frac{2k(1-b)}{b}Q^{-\frac{1}{b}}+\frac{2(b-1)}{b}Q\right)\\
  =&-\frac{4k(1-b^2)(b+2)}{b^3}Q^{-\frac{1}{b}}\alpha+\frac{6(2b+1)(1-b)}{b^2}\alpha-\frac{2(b^2-1)}{kb^2}.
\end{align*}
According to the expressions of $Q$ and $SQ$, we obtain that
$$\left(SQ,Q'\right)_{\alpha}=\frac{2k(1-b)}{b}\int Q^{-\frac{1}{b}}Q'\alpha\,\mathrm{d}x+\frac{2(b-1)}{b}\int QQ'\alpha\,\mathrm{d}x=0$$
which implies that $SQ\in N^{\bot}(\mathcal{L}).$

We finally verify the spectral properties $(\mathcal{L}^{-1}SQ,SQ)<0$. In fact, we deduce from $b<-1$ and $A>1$ that $\left|\frac{2k(1-b)}{b}\right|>\left|\frac{2(b-1)}{b}\right|,$
which combining with the fact $\int Q^{-\frac{1}{b}+2}\,\mathrm{d}x> \int Q^3\,\mathrm{d}x>0$ and $\mathcal{L}(Q^2)=SQ$ yield that
$$\left(\mathcal{L}^{-1}SQ,SQ\right)=\frac{2k(1-b)}{b}\int Q^{-\frac{1}{b}+2}\,\mathrm{d}x+\frac{2(b-1)}{b}\int Q^3\,\mathrm{d}x<0.$$
This completes the proof of Lemma \ref{bfg11}.
\end{proof}
Building on Lemma \ref{bfg11}, we proceed in accordance with the idea in \cite{MM9} to weaken its conditions, thereby proposing the following lemma.
\begin{lemma}\label{bfgg}
  For $\theta>0$, if $\eta$ satisfies
\begin{equation}\label{cpc}
  \left|\left(\eta,\frac{SQ}{\left\|SQ\right\|_{L^2}}\right)\right|+ \left|\left(\eta,\frac{Q'}{\left\|Q'\right\|_{L^2}}\right)\right|\leq \theta\left\|\eta\right\|_{H_{\alpha}^1}, \end{equation} then there holds,
  \begin{equation}\label{bf12}
    (\mathcal{L}\eta,\eta)\geq\frac{3\lambda_1}{4}\left\|\eta\right\|_{H^1_{\alpha}}^2.
  \end{equation}
\end{lemma}
\begin{proof}
  Assume that $\eta\in C\left(\mathbb{R},H_{\alpha}^1\right)\cap L^{\infty}\left(\mathbb{R},H_{\alpha}^1\right)$ satisfies \eqref{cpc} and $\eta_{1}(x,t)$ satisfies the orthogonality conditions $(\eta_{1},SQ)=(\eta_{1},Q')=0.$
  The decomposition of $\eta$ can be given as
  \begin{equation}\label{dec1}
 \eta=\eta_{1}+a_1\frac{SQ}{\left\|SQ\right\|_{L^2}}+a_2\frac{Q'}{\left\|Q'\right\|_{L^2}}:=\eta_{1}+\eta_{2}.
 \end{equation}
It reveals from the definition of $\mathcal{L}$ that
\begin{equation}\label{bde1}
(\mathcal{L}\eta,\eta)=(\mathcal{L}\eta_1,\eta_1)+(\mathcal{L}\eta_2,\eta_2)+2(\mathcal{L}\eta_1,\eta_2).
\end{equation}
We deduce from  \eqref{cpc}, \eqref{dec1} and Lemma \ref{bfg11} that
\begin{equation}\label{ae1}
  |a_1|+|a_2|\leq \theta\left\|\eta\right\|_{H_{\alpha}^1},
\end{equation}
and
\begin{equation}\label{w11}
    (\mathcal{L}\eta_1,\eta_1)\geq\lambda_1\left\|\eta_1\right\|^2_{H_{\alpha}^1}.
\end{equation}
Moreover, it follows from \eqref{dec1} and \eqref{ae1} that
\begin{equation}\label{we1}
 \left\|\eta\right\|^2_{H_{\alpha}^1}\leq  \left\|\eta_1\right\|^2_{H_{\alpha}^1}+\theta\left\|\eta\right\|^2_{H_{\alpha}^1}~~\text{and}~~\left\|\eta_1\right\|^2_{H_{\alpha}^1}\leq\left\|\eta\right\|^2_{H_{\alpha}^1},
\end{equation}
which together with \eqref{w11} to obtain that
\begin{equation}\label{c2.52}
   (\mathcal{L}\eta_1,\eta_1)\geq\frac{7\lambda_1}{8}\left\|\eta\right\|^2_{H_{\alpha}^1}.
\end{equation}
Next, it remains to estimate $(\mathcal{L}\eta_2,\eta_2).$ In view of the properties of $Q$ and \eqref{ae1}, we have
\begin{equation}\label{c2.53}
  \left|(\mathcal{L}\eta_2,\eta_2)\right|\leq|a_1|^2+|a_2|^2\leq
  \theta^2\left\|\eta\right\|^2_{H_{\alpha}^1}\leq\frac{\lambda_1}{16}\left\|\eta\right\|^2_{H_{\alpha}^1}
\end{equation}
and
\begin{equation}\label{c2.54}
  \left|(\mathcal{L}\eta_1,\eta_2)\right|\leq\left\|\eta_1\right\|_{H_{\alpha}^1}\left\|\eta_2\right\|_{H_{\alpha}^1}\leq
  \left\|\eta_1\right\|_{H_{\alpha}^1}\left(|a_1|+|a_2|\right)\leq
  \theta\left\|\eta\right\|^2_{H_{\alpha}^1}\leq\frac{\lambda_1}{16}\left\|\eta\right\|^2_{H_{\alpha}^1}.
\end{equation}
We conclude the proof of \eqref{bf12} follows gathering \eqref{bde1}, \eqref{c2.52}, \eqref{c2.53} and \eqref{c2.54}.
The proof of Lemma \ref{bfgg} is complete.
\end{proof}

Lemma \ref{bfgg} not only constitutes the fundamental cornerstone of the linear Liouville problem, but also directly extends to the nonlinear Liouville theorem without requiring additional modifications by relaxing the orthogonality constraint. Thus, this lemma effectively bridges the gap between linear and nonlinear Liouville theory.

With Lemmas \ref{bfg11} and \ref{bfgg} in hand, we are now prepared to prove Lemma \ref{bfg1}.
\begin{proof}[Proof of Lemma \ref{bfg1}]
 Assume that $\eta(x,t)\in C\left(\mathbb{R},H_{\alpha}^1\right)\cap L^{\infty}\left(\mathbb{R},H_{\alpha}^1\right)$ satisfies the orthogonality conditions \eqref{woc1} and \eqref{woc2}.
 It then transpires from the definition of $L(\eta,\eta)$ in \eqref{bf1} that
 \begin{equation}\label{b1dc}
   L(\eta,\eta)=\left(\mathcal{L}\sqrt{\Phi_M}\eta,\sqrt{\Phi_M}\eta\right)-\int\left(\partial_x\sqrt{\Phi_M}\right)^2\eta^2\alpha\,\mathrm{d}x
   -\int \Phi_M'\eta\eta_x\alpha\,\mathrm{d}x.
 \end{equation}
Provided that verification of  $\sqrt{\Phi_M}\eta$ satisfies \eqref{cpc}, we can obtain the desired coercivity results. It follows from the definition of $\Phi_M(x)$ and the properties of $Q$ that
 \begin{equation}\label{ver11}
   \left|\int\sqrt{\Phi_M}\eta\frac{SQ}{\left\|SQ\right\|_{L^2}}\,\mathrm{d}x\right|\leq\theta\left\|\sqrt{\Phi_M}\eta\right\|_{L_{\alpha}^2}
 \end{equation}
 and
 \begin{equation}\label{ver12}
   \left|\int\sqrt{\Phi_M}\eta\frac{Q'}{\left\|Q'\right\|_{L^2}}\,\mathrm{d}x\right|\leq\theta\left\|\sqrt{\Phi_M}\eta\right\|_{L_{\alpha}^2},
 \end{equation}
which in turn imply that
\begin{equation}\label{c2.58}
  \left(\mathcal{L}\sqrt{\Phi_M}\eta,\sqrt{\Phi_M}\eta\right)\geq \frac{3\lambda_1}{4}\left\|\sqrt{\Phi_M}\eta\right\|_{H_{\alpha}^1}^2.
\end{equation}
For the last two terms of the right hand of \eqref{b1dc}, invoking the properties of $\Phi_M(x)$ in \eqref{gine}, we get that
\begin{equation}\label{c2.59}
 \left|\int\left(\partial_x\sqrt{\Phi_M}\right)^2\eta^2\alpha\,\mathrm{d}x\right|\leq\frac{C}{M^2}\int \eta^2\alpha\Phi_M\,\mathrm{d}x\leq\frac{\lambda_1}{8}\int \eta^2\alpha\Phi_M\,\mathrm{d}x,
\end{equation}
and
\begin{equation}\label{c2.60}
\left|\int \Phi_M'\eta\eta_x\alpha\,\mathrm{d}x\right|\leq\frac{C}{M^2}\int (\eta^2\alpha+\eta_x^2\alpha)\Phi_M\,\mathrm{d}x\leq\frac{\lambda_1}{8}\int (\eta^2\alpha+\eta_x^2\alpha)\Phi_M\,\mathrm{d}x.
\end{equation}
We conclude the proof of \eqref{bfw1} from \eqref{b1dc} and \eqref{c2.58}--\eqref{c2.60}. Here, we can choose $\kappa=\lambda_1/2,$ where $\lambda_1$ is given in Lemma \ref{bfg11}. This completes the proof of Lemma \ref{bfg1}.
\end{proof}
We are now in the position to obtain the {\bf coercivity property}  in \eqref{lzze}.
\begin{proof}[Proof of \eqref{lzze}]
To prove \eqref{lzze}, it remains to consider the last three terms on the right hand of \eqref{lz}. We first note that
$$\left|\frac{4b(1+2b)}{M(b+1)^2}\tanh(\nu x)\right|\leq \frac{4(1-b)b^2}{M(b+1)^2}\quad\text{for}\quad b\in(-2,-1)$$
and
$$\left|\frac{4b(1+2b)}{M(b+1)^2}\tanh(\nu x)\right|\leq 4A\frac{(1-b)}{2M}b^2\quad\text{for}\quad b\leq-2,$$
so that we choose $M=\max\left\{1,\frac{32A(1-b)}{\lambda_1(b+1)^2},\frac{32A(1-b)}{\lambda_1}\right\}$ for $b<-1.$  We deduce from \eqref{vgf}--\eqref{gine} that
\begin{align}\label{3lz}
  \left|-\frac{k}{b^2}\int\eta^2\Phi_M'\alpha_x\,\mathrm{d}x-\frac{k}{b^2}\int\eta^2\Phi_M''\alpha\,\mathrm{d}x+\frac{k}{2b^2}\int\eta^2\frac{(\Phi_M')^2}{\Phi_M}\alpha\,\mathrm{d}x\right|\leq
 \frac{\lambda_1}{4}\int\eta^2\alpha\Phi_M\,\mathrm{d}x.
\end{align}
 We deduce gathering \eqref{bfw1} in Lemma \ref{bfg1} and \eqref{3lz} that \eqref{lzze} holds true with $\lambda=\lambda_1/4$.
\end{proof}

To proceed, we need to verify the following.\\
 For $\frac12<s<1$ and $b<-1,$ there exists a positive constant $C$ such that
\begin{align}\label{wlz}
    \left|\frac12\frac{1}{(M+t^2)^{s}}\frac{\mathrm{d}}{\mathrm{d}t}\int \eta^2\phi\,\mathrm{d}x+\left(\mathcal{L}z,z\right)\right|\leq \frac{C}{(M+t^2)^{s}}\left\|\eta\right\|^2_{H^1_{\alpha}}.
\end{align}
In fact,  in viewing of \eqref{w1}, for $\lambda_1>0$, it reveals from the Cauchy-Schwarz inequality that
\begin{align}\label{IE1}
  \left|\frac{k(2b-1)}{b(1-b)}\int Q^{-\frac{1}{b}-1}Q'\eta^2\phi\,\mathrm{d}x \right|\leq\frac{k(2b-1)}{b(1-b)^2}\left|\int Q'Q\eta^2\alpha\phi\,\mathrm{d}x \right|\leq\frac{\lambda_1}{16}\left\|\eta\right\|^2_{H^1_{\alpha}},
\end{align}
where we note that $\frac{k(2b-1)}{b(1-b)^2}\leq\frac{3}{4}$ and $|Q|\leq A\frac{1-b}{2}\leq\frac{\lambda_1}{16}.$ In addition, we also get that
\begin{align}\label{IE2}
  \left|\frac{k}{(1-b)}\int Q^{-\frac{1}{b}}\eta^2\phi'\,\mathrm{d}x \right|\leq\frac{k}{(1-b)^2}\left|\int Q^2\eta^2\alpha\phi'\,\mathrm{d}x \right|\leq\frac{\lambda_1}{16}\left\|\eta\right\|^2_{H^1_{\alpha}},
\end{align}
and
\begin{align}\label{IE3}
  \left|\int \eta\phi(1-\partial_x^2)^{-1}\left(bQ\partial_x\eta+(b-1)Q'\eta\right)\,\mathrm{d}x \right|\leq\frac{\lambda_1}{8}\left\|\eta\right\|_{L^{\infty}}\int\eta^2\phi^2\,\mathrm{d}x\leq
  \frac{\lambda_1}{4}\left\|\eta\right\|^2_{H^1_{\alpha}},
\end{align}
where $\phi^2=(Q')^2=b^2\left(Q^2-Q^{\frac{1}{b}+3}/k\right)\lesssim \Phi$.
The proof of \eqref{wlz} is inferred gathering \eqref{w1}, \eqref{lz}, \eqref{3lz} and \eqref{IE1}--\eqref{IE3}.

With \eqref{lzze} and \eqref{wlz} in hand, we can conclude the proof of Theorem \ref{llt}.
\begin{proof}[Proof of Theorem \ref{llt}]
Invoking \eqref{lzze} and \eqref{wlz}, there exists a positive constant $C_{b,\lambda_1}$ such that
\begin{equation}\label{c2.72}
  -\frac12\frac{1}{(M+t^2)^{s}}\frac{\mathrm{d}}{\mathrm{d}t}\int \eta^2\phi\,\mathrm{d}x\geq\frac{\lambda_1}{4}\int(\eta^2\alpha+\eta_x^2\alpha)\Phi_M\,\mathrm{d}x-\frac{C_{b,\lambda_1}}{(M+t^2)^{s}}\left\|\eta\right\|^2_{H^1_{\alpha}}.
\end{equation}
Integrating \eqref{c2.72} and using the boundedness of $\phi(x),$ we obtain that
\begin{equation}\label{c2.73}
  \int_{\mathbb{R}}\int \eta^2(x,t)\Phi_M(x)\,\mathrm{d}x\,\mathrm{d}t\leq\int_{\mathbb{R}}\int \eta^2(x,t)\alpha(x)\Phi_M(x)\,\mathrm{d}x\,\mathrm{d}t\leq C_{b,M,\lambda_1}\sup_t\|\eta(\cdot,t)\|^2_{L^2}<+\infty.
\end{equation}
Thus, there exists a time sequence $t_n\to\infty~(n\to\infty)$ such that
\begin{equation}\label{c2.74}
  \int \eta^2(x,t_n)\Phi_M\,\mathrm{d}x\to0,
\end{equation}
as $n\to\infty.$ Moreover, for some $A>0,$ one has $|\Phi_M(x)|\geq C_{A,b}>0$ if $|x|\leq A$.  Then,  we derive from \eqref{hode} and \eqref{c2.74} that
\begin{align}\label{c2.75}
  \int \eta^2(x,t_n)\,\mathrm{d}x=&\int_{|x|>A} \eta^2(x,t_n)\,\mathrm{d}x+\int_{|x|\leq A} \eta^2(x,t_n)\,\mathrm{d}x\nonumber\\
  \leq&\frac{1}{C_A}\int \eta^2(x,t_n)\Phi_M(x)\,\mathrm{d}x+\epsilon,
\end{align}
which leads to
\begin{equation}\label{c2.76}
  \int \eta^2(x,t_n)\,\mathrm{d}x\to0~~~~\text{as}~n\to\infty.
\end{equation}
In a similar way, for a sequence $s_n\to\infty,$ we get that
\begin{equation}\label{c2.77}
  \int \eta^2(x,-s_n)\,\mathrm{d}x\to0~~~~\text{as}~n\to\infty.
\end{equation}
Integrating \eqref{c2.72} on $(-s_n,t_n)$ and using \eqref{c2.76} and \eqref{c2.77}, we have
\begin{equation}\label{c2.78}
  \int_{-s_n}^{t_n}\int \eta^2(x,t)\Phi_M\,\mathrm{d}x\,\mathrm{d}t\leq \frac{C}{\lambda_1}\left(\int \eta^2(x,t_n)\,\mathrm{d}x+\int \eta^2(x,-s_n)\,\mathrm{d}x\right),
\end{equation}
which leads to
\begin{equation}\label{c2.79}
   \int_{-s_n}^{t_n}\int \eta^2(x,t)\Phi_M\,\mathrm{d}x\,\mathrm{d}t=0~~~~\text{as}~n\to\infty.
\end{equation}
Notice that the positive characteristics of $\Phi_M(x)$ implies that $\eta(x,t)=0$ for $(x,t)\in\mathbb{R}^2.$
The definitions of $\eta(x,t)$ in \eqref{dp} gives that there exists $a_0(t), b_0(t)\in\mathbb{R}$  such that
\begin{equation}\label{c2.80}
v=a_0(t)Q'+b_0(t)Q.
\end{equation}
Moreover, it is  worth noting that $\mathcal{B}(Q)\mathcal{L}Q=b(b-1)Qq'+(b-1)Q'q=0,$
relying on \eqref{ker} and the definition of $v(x,t)$ in \eqref{l2c}, we have
$$a_0'Q'+b_0'Q=\frac{1}{1-b}\mathcal{B}(Q)\mathcal{L}\left(a_0Q'+b_0Q\right)=0.$$
This implies that $a_0'(t)=b_0'(t)=0,$ that is $v(x,t)=a_0Q'+b_0Q$ for two real constants $a_0,b_0.$
This completes the proof of Theorem \ref{llt}.
\end{proof}

\section{Nonlinear Liouville property}\label{s3}
The focus  in this section is  to prove  the nonlinear Liouville property Theorem \ref{nlt}.  We start with a modulation argument (Lemma \ref{zlem3.1}) for solutions close to the leftons, we then construct a space transition parameter $\rho(t)$ such that the error \eqref{z3.3} satisfies some orthogonal conditions and is bounded in $H^1_{\alpha}$. Next, we derive monotonicity property for the localized version of conservation laws in Lemma \ref{uvm1} and Lemma \ref{mener}, a direct consequence of which is the uniformly exponential decay of the solutions and the error.  We analyze the localized error with transformed variable and present the coercivity property of the quadratic form \eqref{lwe}. Gathering with \eqref{wlw}, we show that the localized error vanishes, implying that the nonlinear Liouville property holds true.

Let us start with the modulation of solution of $b$-family of equations close to the leftons.
\begin{lemma}\label{zlem3.1}
 Let $\epsilon_0>0,K_0>0$ and $0<\epsilon \leq \epsilon_0$. Assume that $m(x,t)$ is the global solution of $b$-family of equations with the initial data $0<m_0(x)\in\mathcal{Z}$ satisfying
\begin{equation}\label{z3.2}
\inf _{y \in \mathbb{R}}\|m(\cdot, t)-\gamma Q(\cdot-y)\|_{H^1} \leq \epsilon.
\end{equation}
Then there exists a function  $\rho(t)\in C^1\left(\mathbb{R}\right)$ such that
\begin{equation}\label{z3.3}
\varepsilon(x, t)= m(x+\rho(t), t)-\gamma Q(x)
\end{equation}
satisfies
\begin{align}
\left(\varepsilon,Q'\right)_{\alpha}=\left(\varepsilon,F_2'(Q)\right)_{\alpha}=0,\label{voc}
\end{align}
and
\begin{align}
&\|\varepsilon(\cdot, t)\|_{H_{\alpha}^1} \leq K_0 \epsilon,\label{z3.6}
\end{align}
where $\gamma$ is defined in Theorem \ref{os} and $K_0>0$. Moreover, for some constant $K_1>0$ and $t\in\mathbb{R}$, we have
\begin{equation}\label{z3.7}
\left|\rho^{\prime}(t)\right|\leq K_1\|\varepsilon(\cdot, t)\|_{H^1}.
\end{equation}
\end{lemma}
\begin{proof}
  The proof relies on standard arguments from the implicit function theorem. Let us define
  \begin{equation}\label{pmap}
    U(m,y_1)=\left((\varepsilon_1,Q')_{\alpha},(\varepsilon_1,F_2'(Q))_{\alpha}\right),
  \end{equation}
  where $\varepsilon_1=m(x+y_1(t),t)-\gamma Q(x).$ Then the map $U$ is defined from $\mathcal{Z}(\mathbb{R})\times \mathbb{R}$ to $\mathbb{R}^2.$ It reveals form \eqref{pmap} that
  $U(\gamma Q,0)=0.$ Relying on the definition of $Q,$ one deduces that
  \begin{align*}
    &\partial_{y_1}U(m,y_1)\left|_{m=\gamma Q,y_1=0}\right.=\gamma\left((Q',Q')_{\alpha},(Q',F_2'(Q))\right)\\
    =&\gamma\left(\left(A\frac{1-b}{2}\right)^{-\frac{1}{b}}\frac{(b+1)^2}{4}\int \cosh^{\frac{-10b-4}{b+1}}x\sinh^2x\,\mathrm{d}x, \frac{-b-1}{2k^2}\int \cosh x\sinh x\,\mathrm{d}x\right).
  \end{align*}
Following the implicit function theorem, there exists $\alpha_1>0$, $D$ a neighbourhood of $0$ in $\mathbb{R}$ and a $C^1$ map $y_1:\{m\in \mathcal{Z}(\mathbb{R}),\|m-\gamma Q\|_{H^1}\leq\epsilon_1\}\to D$ such that $U(m,y_1(m))=0$ and
 $$
|y_1(m)|\leq C\|m-\gamma Q\|_{H^1}.
  $$
We uniquely extend $C^1$ map $y_1(m)$ to $D_{\epsilon_1}=\{m\in \mathcal{Z}, \, \inf_{\xi}\|m(\cdot+\xi)-Q\|_{H^1}\leq\epsilon_1\}$ such that for all $m$ and $\xi,$
$y_1(m)=y(m(\cdot+\xi(m)))+\xi(m).$ Thus, we set
$\varepsilon_1=m(x+y_1(m),t)-\gamma Q(x)$ such that
$(\varepsilon_1(m), Q(x))_{\alpha}=(\varepsilon_1(m), F_2'(Q) )=0$
and
 \begin{equation}\label{pest}
    \|\varepsilon_1(m)\|_{\mathcal{Z}(\mathbb{R})}\leq C\|m-\gamma Q\|_{H^1}.
  \end{equation}
Then, setting $\rho(t)=y_1(m(t)),\varepsilon(x,t)=\varepsilon_1(m),$ it in turn implies that \eqref{z3.6} holds.

To prove \eqref{z3.7}, we deduce from the equations for $m$ and $u$ that
\begin{align}\label{vart}
    \varepsilon_t=&\partial_x\left(q\partial_x^2\tilde{\varepsilon}-(b+1)q\tilde{\varepsilon}+(b-1)q'\partial_x\tilde{\varepsilon}+q''\tilde{\varepsilon}\right)\nonumber\\
    &+\partial_x\left(\frac{b-1}{2}\left((\partial_x\tilde{\varepsilon})^2-\tilde{\varepsilon}^2\right)-\tilde{\varepsilon}\varepsilon\right)+\rho'(t)\left(\gamma Q'+\partial_x\varepsilon\right)\nonumber\\
    =&\frac{1}{1-b}\mathcal{B}(Q)\mathcal{L}\varepsilon+\partial_x\left(\frac{b-1}{2}\left((\partial_x\tilde{\varepsilon})^2-\tilde{\varepsilon}^2\right)-\tilde{\varepsilon}\varepsilon\right)+\rho'(t)\left(\gamma Q'+\partial_x\varepsilon\right),
\end{align}
where $\tilde{\varepsilon}:=\left(1-\partial_x^2\right)^{-1}\varepsilon,~~q=\left(1-\partial_x^2\right)^{-1}Q.$ Thus, we differentiate the orthogonality conditions in \eqref{voc} to get
\begin{align}\label{mod1}
\int\frac{1}{1-b}\mathcal{B}(Q)\mathcal{L}\varepsilon Q'\alpha\,\mathrm{d} x&+\frac{b-1}{2}\int\partial_x\left(\frac{b-1}{2}\left((\partial_x\tilde{\varepsilon})^2-\tilde{\varepsilon}^2\right)-\tilde{\varepsilon}\varepsilon\right)Q'\alpha\,\mathrm{d} x\nonumber\\
&+\rho'(t)\int\left(\gamma Q'+\partial_x\varepsilon\right)Q'\alpha\,\mathrm{d} x=0,
\end{align}
which leads to
\begin{equation}\label{c3.13}
  |\rho'(t)|=\frac{\left|\int\frac{1}{1-b}\mathcal{B}(Q)\mathcal{L}\varepsilon Q'\alpha\,\mathrm{d} x-\frac{b(b^2-1)}{4k}\int\left((\partial_x\tilde{\varepsilon})^2-\tilde{\varepsilon}^2-\tilde{\varepsilon}\varepsilon\right)\,\mathrm{d} x\right|}{\left|\int \gamma (b^2Q^{-\frac{1}{b}}-\frac{b^2}{k}Q)+\varepsilon_x Q'\alpha\,\mathrm{d} x\right|}.
\end{equation}
Recalling the the fact that $\int Q^{-\frac{1}{b}}-\frac{1}{k}Q>0,$ we obtain
$$|\rho'(t)|\lesssim\frac{\int \varepsilon^2 e^{-\nu|x|}\,\mathrm{d}x}{b^2|\gamma|\int Q^{-\frac{1}{b}}-\frac{1}{k}Q\,\mathrm{d}x-K_0\epsilon}\lesssim K_1\|\varepsilon\|_{H^1}.$$
This finishes the proof of the modulation of solutions, which  completes the proof of Lemma \ref{zlem3.1}.
\end {proof}

\subsection{Monotonicity and smoothness}
Prior to proving the nonlinear Liouville theorem, we begin with the monotonicity property of the solution $m(x,t)$. Leveraging the modulation argument in Lemma \ref{zlem3.1}, it follows that there exists $\rho(t)\in C^1(\mathbb{R})$ such that \eqref{z3.3}--\eqref{z3.7} hold true.
We recall that $\psi_L(x)\in C^{\infty}(\mathbb{R})$ defined as in \eqref{tf}--\eqref{psi} and fix $L\geq4$.
For $x_0>0$ and $t_0\in\mathbb{R},$ we define the quantities
$$I(t):=I_{x_0,t_0}(t)=\int m^{-\frac{1}{b}}\psi_L(x_1)\,\mathrm{d}x+\frac{1}{b^2}\int m^{-\frac{1}{b}-2}m_x^2\psi_L(x_1)\,\mathrm{d}x$$
where $x_1=x-\rho(t_0)+4b(t-t_0)-x_0$ for $t_0\leq t.$

\begin{lemma}\label{uvm1}
  Let $m\in C(\mathbb{R},H_{\alpha}^1(\mathbb{R}))$ be the solutions of \eqref{bfe} satisfying \eqref{lcuv}. For any $x_0>0$, $t_0\in\mathbb{R}$ and $t_0\leq t_1,$ there holds
  \begin{equation}\label{Iuv}
    I(t_1)-I(t_0)\lesssim e^{-x_0/L}.
  \end{equation}
\end{lemma}
\begin{proof}
  By using integration by parts, we deduce from \eqref{bfe} that (the details see Appendix),
\begin{align}
    \frac{\mathrm{d}}{\mathrm{d}t}I(t)=&-\int m^{-\frac{1}{b}}u\psi_L'(x_1)\,\mathrm{d}x+\frac{1}{b^2}\int m^{-\frac{1}{b}-2}m_x^2u\psi_L'(x_1)\,\mathrm{d}x+\frac{2}{1-b}\int m^{-\frac{1}{b}+1}\psi_L'(x_1)\,\mathrm{d}x\nonumber\\
    &+4b\int m^{-\frac{1}{b}}\psi_L'(x_1)\,\mathrm{d}x+\frac{4}{b}\int m^{-\frac{1}{b}-2}m_x^2\psi_L'(x_1)\,\mathrm{d}x.\label{i1}
  \end{align}

  Next, we consider the first term of the right hand of \eqref{i1}. Let $R_0$ be chosen later. For $\int m^{-\frac{1}{b}}u\psi_L'(x_1)\,\mathrm{d}x=\int m^{-\frac{1}{b}}\psi_L'(x_1)\left(u-\gamma q(\cdot-\rho)+\gamma q(\cdot-\rho)\right)\,\mathrm{d}x,$ there are two cases for $|x-\rho(t)|.$ In the case $|x-\rho(t)|\geq R_0,$ invoking the properties of $q(x)$ and \eqref{z3.6}, we obtain, for some $\delta_0>0,$
  $$\left|\int m^{-\frac{1}{b}}\psi_L'(x_1)q(\cdot-\rho)\,\mathrm{d}x\right|\leq Ce^{-\delta_0R_0}\int m^{-\frac{1}{b}}\psi_L'(x_1)\,\mathrm{d}x.$$
  In the case $|x-\rho(t)|\leq R_0,$ we have
  $$\left|\int m^{-\frac{1}{b}}\psi_L'(x_1)q(\cdot-\rho)\,\mathrm{d}x\right|\leq Ce^{(R_0+4b(t-t_0)-x_0)/L}\int m^{-\frac{1}{b}}\,\mathrm{d}x,$$
  where $|\psi'_L(x_1)|\lesssim e^{x_1/L}.$ In addition, it follows from \eqref{z3.6} that
  $$\left|\int m^{-\frac{1}{b}}\psi_L'(x_1)(u-\gamma q(\cdot-\rho))\,\mathrm{d}x\right|\leq K_0\epsilon\int m^{-\frac{1}{b}}\psi_L'(x_1)\,\mathrm{d}x.$$
For the second and third terms of the right hand of \eqref{i1}, we similarly obtain that
\begin{align*}
    \left|\int m^{-\frac{1}{b}-2}m_x^2u\psi_L'(x_1)\,\mathrm{d}x\right|\lesssim& \left(K_0\epsilon+e^{-\delta_0R_0}\right)\int m^{-\frac{1}{b}-2}m_x^2\psi_L'(x_1)\,\mathrm{d}x\\
    &+e^{(R_0+4b(t-t_0)-x_0)/L}\int m^{-\frac{1}{b}-2}m_x^2\,\mathrm{d}x,
\end{align*}
and
\begin{align*}
    \left|\int m^{-\frac{1}{b}+1}\psi_L'(x_1)\,\mathrm{d}x\right|&\lesssim \left|\int m^{-\frac{1}{b}}\left(m-\gamma Q(\cdot-\rho)+\gamma Q(\cdot-\rho)\right)\psi_L'(x_1)\,\mathrm{d}x\right|\\
    &\lesssim\left(K_0\epsilon+e^{-\delta_0R_0}\right)\int m^{-\frac{1}{b}}\psi_L'(x_1)\,\mathrm{d}x
    +e^{(R_0+4b(t-t_0)-x_0)/L}\int m^{-\frac{1}{b}}\,\mathrm{d}x.
\end{align*}

  Therefore, for $\epsilon$ sufficiently small, $R_0$ large and the assumption on conservation law in the well-posedness theory, it is deduced that
  \begin{align}\label{i11}
    \frac{\mathrm{d}}{\mathrm{d}t}I(t)\leq& 2b\int m^{-\frac{1}{b}}\psi_L'(x_1)\,\mathrm{d}x+\frac{2}{b}\int m^{-\frac{1}{b}-2}m_x^2\psi_L'(x_1)\,\mathrm{d}x\nonumber\\
    &+Ce^{(R_0+4b(t-t_0)-x_0)/L}.
  \end{align}

Integrating \eqref{i11} on $(t_0-1,t_0)$, we obtain that, for a positive  constant $C>0,$
\begin{equation*}
    I(t_0)-I(t_0-1)\leq Ce^{- x_0/L}~~,\forall t_0\in\mathbb{R},
  \end{equation*}
  where $I(t_0-1)=\int (m^{-\frac{1}{b}}(x,t_0-1)+\frac{1}{b^2}m^{-\frac{1}{b}-2}(x,t_0-1))m_x^2(x,t_0-1)))\psi_L(x_1(t_0-1))\,\mathrm{d}x\lesssim1.$
Similarly, we integrate \eqref{i11} on $[t_0,t_0+\tau]$, for any $\tau>0,$ to get
\begin{equation*}
    I(t_0+\tau)-I(t_0)\leq Ce^{- x_0/L}e^{-(t_0+\tau)/L}.
  \end{equation*}
Taking the limit $\tau\to\infty,$ we deduce from \eqref{lcuv} that
$$I(t_0+\tau)=\int (m^{-\frac{1}{b}}+\frac{1}{b^2}m_x^2)(x,t_0+\tau)\psi_L(x+4b\tau-x_0)\,\mathrm{d}x\to0$$
and $I(t_0)\geq-Ce^{-\nu x_0/L},$ which implies that $|I(t_0)|\lesssim e^{-x_0/L}.$
In addition, integrating \eqref{i11} on $(\tilde{t}_0,\tilde{t}_0+1)$, we get that
\begin{align}
\int&\left(m^{-\frac{1}{b}}+\frac{1}{b^2}m^{-\frac{1}{b}-2}m_x^2\right)(x,t_0)\psi_L(x-\rho(t_0)-x_0)\,\mathrm{d}x\nonumber\\
&+\int_{\tilde{t}_0}^{\tilde{t}_0+1}\int\left(m^{-\frac{1}{b}}+\frac{1}{b^2}m^{-\frac{1}{b}-2}m_x^2\right)\psi'_L(x_1)\,\mathrm{d}x\,\mathrm{d}t
    \lesssim e^{- x_0/L},\label{i12}
\end{align}
for $ \tilde{t}_0\in\mathbb{R}.$ This completes the proof Lemma \ref{uvm1}.
\end{proof}
To establish the smoothness and exponential decay of the solution of \eqref{bfe}, we present monotonicity arguments for $\varepsilon$.
\begin{lemma}\label{mener}
  For $b<-1,$ let $\varepsilon(x,t)$ satisfy the assumptions of Lemma \ref{zlem3.1} and \eqref{z3.3}--\eqref{z3.7}. For any $x_0>0$, $t_0\in\mathbb{R}$ and $t_0\leq t,$ we define
  \begin{equation}\label{j1}
    \mathcal{E}(t)=\int\varepsilon^2\psi_L(\tilde{x})\,\mathrm{d}x,
  \end{equation}
  where $\psi_L$ is defined in \eqref{tf}--\eqref{psi} and $\tilde{x}=x+4b(t-t_0)-x_0$.
  Then, we have
  \begin{equation}\label{juv}
    \mathcal{E}(t_1)-\mathcal{E}(t_0)\lesssim e^{-x_0/L},
  \end{equation}
  where $t_0\leq t_1.$
\end{lemma}
\begin{proof}
  We compute \eqref{vart} through integration by parts, then one deduces that, for $\varepsilon=\left(1-\partial_x^2\right)\tilde{\varepsilon},$
 \begin{align}\label{varve}
    \frac{\mathrm{d}}{\mathrm{d}t}\int\varepsilon^2\psi_L(\tilde{x})\,\mathrm{d}x=&2\int\varepsilon\varepsilon_t\psi_L(\tilde{x})\,\mathrm{d}x+4b\int\varepsilon^2\psi_L'(\tilde{x})\,\mathrm{d}x\nonumber\\
=&2\int\varepsilon\psi_L(\tilde{x})\left(-bq'\varepsilon-q\varepsilon_x-bQ\tilde{\varepsilon}_x-Q'\tilde{\varepsilon}\right)\,\mathrm{d}x\nonumber\\
&+2\int\varepsilon\psi_L(\tilde{x})\partial_x\left(\frac{b-1}{2}\left((\partial_x\tilde{\varepsilon})^2-\tilde{\varepsilon}^2\right)-\tilde{\varepsilon}\varepsilon\right)\,\mathrm{d}x\nonumber\\
&+2\rho'(t)\int\varepsilon\psi_L(\tilde{x})\left(\gamma Q'+\partial_x\varepsilon\right)\,\mathrm{d}x+4b\int\varepsilon^2\psi_L'(\tilde{x})\,\mathrm{d}x.
\end{align}
Arguing as in Lemma \ref{lmp1}, we have, for some $R_0>0,$
\begin{align}\label{var1}
\left|(-b)\int\varepsilon\varepsilon_xq\psi_L(\tilde{x})\,\mathrm{d}x\right|+\left|\int\varepsilon^2q'\psi_L(\tilde{x})\,\mathrm{d}x\right|\lesssim &e^{(R_0+4b(t-t_0)-x_0)/L}\int\varepsilon^2\,\mathrm{d}x\nonumber\\
&+(-b)\operatorname{sech}(\nu R_0/L)\int\varepsilon^2\psi_L'(\tilde{x})\,\mathrm{d}x.
\end{align}
Moreover, the decay properties of $Q$ and $q$ ensure that
\begin{align}\label{var3}
    \left|\int\varepsilon\tilde{\varepsilon}_x Q\psi_L(\tilde{x})\,\mathrm{d}x\right|\lesssim& \|\tilde{\varepsilon}_x\|_{L^{\infty}}\left(\int\varepsilon^2 Q\psi_L(\tilde{x})\,\mathrm{d}x\right)^{1/2}\nonumber\\
    \lesssim&e^{(R_0+4b(t-t_0)-x_0)/2L}\|\varepsilon\|_{L^2}+\sqrt{\operatorname{sech}(\nu R_0/L)}\int\varepsilon^2\psi_L'(\tilde{x})\,\mathrm{d}x,
\end{align}
and
\begin{align}\label{var4}
    \left|\int\varepsilon\tilde{\varepsilon} Q'\psi_L(\tilde{x})\,\mathrm{d}x\right|\lesssim& \int(\varepsilon^2+\tilde{\varepsilon}^2) Q\psi_L(\tilde{x})\,\mathrm{d}x\nonumber\\
    \lesssim&e^{(R_0+4b(t-t_0)-x_0)/2L}\int(\varepsilon^2+\tilde{\varepsilon}^2)\,\mathrm{d}x+\operatorname{sech}(\nu R_0/L)\int(\varepsilon^2+\tilde{\varepsilon}^2) \psi_L'(\tilde{x})\,\mathrm{d}x.
\end{align}
Now, we consider the second term on the right hand of \eqref{varve}. A direct computation reveals that
\begin{align}\label{var5}
  \int\varepsilon\psi_L(\tilde{x})\partial_x\left(\frac{b-1}{2}\left((\partial_x\tilde{\varepsilon})^2-\tilde{\varepsilon}^2\right)-\tilde{\varepsilon}\varepsilon\right)\,\mathrm{d}x=&\left(-b-\frac12\right)\int\varepsilon^2\tilde{\varepsilon}_x \psi_L(\tilde{x})\,\mathrm{d}x\nonumber\\
  &+\frac{1}{2}\int\varepsilon^2\tilde{\varepsilon} \psi_L'(\tilde{x})\,\mathrm{d}x.
\end{align}
For the first term of the right hand of \eqref{var5}, we deduce from \eqref{z3.6} that
\begin{align*}
    \left|\int\varepsilon^2\tilde{\varepsilon}_x \psi_L(\tilde{x})\,\mathrm{d}x\right|\lesssim& \|\tilde{\varepsilon}_x\|_{L^{\infty}}\int\varepsilon\sqrt{\psi_L'(\tilde{x})} \varepsilon\frac{\psi_L(\tilde{x})}{\sqrt{\psi_L'(\tilde{x})}}\,\mathrm{d}x\nonumber\\
    \lesssim& K_0\epsilon\left\|\varepsilon\sqrt{\psi_L'(\tilde{x})}\right\|_{L^2}\int\varepsilon^2\alpha\frac{\psi_L^2(\tilde{x})}{\psi_L'(\tilde{x})}\alpha^{-1}\,\mathrm{d}x
    \lesssim K_0\epsilon\left\|\varepsilon\sqrt{\psi_L'(\tilde{x})}\right\|_{L^2},
\end{align*}
which leads to
\begin{align}\label{var6}
  \left|\int\varepsilon\psi_L(\tilde{x})\partial_x\left(\frac{b-1}{2}\left((\partial_x\tilde{\varepsilon})^2-\tilde{\varepsilon}^2\right)-\tilde{\varepsilon}\varepsilon\right)\,\mathrm{d}x\right|\lesssim K_0\epsilon\left\|\varepsilon\sqrt{\psi_L'(\tilde{x})}\right\|_{L^2}.
\end{align}
It remains to consider the third term of the right hand of \eqref{varve}. We easily check that
\begin{equation}\label{var7}
    \rho'(t)\int\varepsilon\psi_L(\tilde{x})\left(\gamma Q'+\partial_x\varepsilon\right)\,\mathrm{d}x\leq K_0\epsilon e^{(R_0+4b(t-t_0)-x_0)/2L}\|\varepsilon\|_{L^2}+\left(K_0\epsilon-\rho'(t)/2\right)\int\varepsilon^2\psi_L'(\tilde{x})\,\mathrm{d}x.
\end{equation}
It is worth noting that
\begin{align*}
    \int\varepsilon^2\psi_L'(\tilde{x})\,\mathrm{d}x=&\int\left(\tilde{\varepsilon}-\tilde{\varepsilon}_{xx}\right)^2\psi_L'(\tilde{x})\,\mathrm{d}x\nonumber\\
    =&\int\left(\tilde{\varepsilon}^2+2\tilde{\varepsilon}_{x}^2+\tilde{\varepsilon}_{xx}^2\right)\psi_L'(\tilde{x})\,\mathrm{d}x-\int\tilde{\varepsilon}^2\psi_L'''(\tilde{x})\,\mathrm{d}x,
\end{align*}
which implies that
\begin{align}\label{var8}
\int\left(\tilde{\varepsilon}^2+\tilde{\varepsilon}_{x}^2+\tilde{\varepsilon}_{xx}^2\right)\psi_L'(\tilde{x})\,\mathrm{d}x\leq C\int\varepsilon^2\psi_L'(\tilde{x})\,\mathrm{d}x.
\end{align}
for $L>4.$ We deduce from \eqref{varve} and \eqref{var1}--\eqref{var8} that
\begin{align}\label{varve1}
    \frac{\mathrm{d}}{\mathrm{d}t}\int\varepsilon^2\psi_L(\tilde{x})\,\mathrm{d}x\lesssim 2b\int\varepsilon^2\psi_L'(\tilde{x})\,\mathrm{d}x+e^{(R_0+4b(t-t_0)-x_0)/L}.
\end{align}
where $R_0$ is large enough such that $\operatorname{sech}(R_0/L)\leq \frac{1}{4}.$
Integrating \eqref{varve1} on $(t_0-1,t_0)$, we obtain that, for a positive  constant $C>0,$
\begin{equation*}
    \mathcal{E}(t_0)-\mathcal{E}(t_0-1)\leq Ce^{- x_0/L},
  \end{equation*}
  where $\mathcal{E}(t_0-1)=\int \varepsilon^2(x,t_0-1)\psi_L(\tilde{x}(t_0-1))\,\mathrm{d}x\lesssim1.$
 We similarly integrate \eqref{varve1} on $[t_0,t_0+\tau]$, for any $\tau>0,$ to get
\begin{equation*}
    \mathcal{E}(t_0+\tau)-\mathcal{E}(t_0)\leq Ce^{- x_0/L}e^{-(t_0+\tau)/L}.
  \end{equation*}
We deduce from the properties of $\psi_L$ that
$$\mathcal{E}(t_0+\tau)=\int \varepsilon^2(x,t_0+\tau)\psi_L(x+4b\tau-x_0)\,\mathrm{d}x\underset{t \to \infty}{\longrightarrow}0,$$
and $\mathcal{E}(t_0)\geq-Ce^{-\nu x_0/L},$ which implies that $|\mathcal{E}(t_0)|\lesssim e^{-x_0/L}.$
In addition, for $ \tilde{t}_0\in\mathbb{R},$ integrating \eqref{varve1} on $(\tilde{t}_0,\tilde{t}_0+1)$, we get that
\begin{align}\label{vardc}
\int&\varepsilon^2(x,t_0)\psi_L(x-x_0)\,\mathrm{d}x+\int_{\tilde{t}_0}^{\tilde{t}_0+1}\int\varepsilon^2(x,t)\psi'_L(\tilde{x})\,\mathrm{d}x\,\mathrm{d}t
    \lesssim e^{- x_0/L}.
\end{align} This completes the proof of Lemma \ref{mener}.
\end{proof}

\begin{corollary}\label{cro3.1}
  Let $m\in C(\mathbb{R},H_{\alpha}^1(\mathbb{R}))$ be the solutions of \eqref{bfe} satisfying \eqref{lcuv}. For a positive constant $\delta_0$, it holds that
  \begin{equation}\label{vml1}
    \sup_{t\in\mathbb{R}}\int\left(m^{-\frac{1}{b}}+\frac{1}{b^2}m^{-\frac{1}{b}-2}m_x^2\right)(x+\rho(t),t)e^{\delta_0|x|}\,\mathrm{d}x\lesssim1,
  \end{equation}
  and
  \begin{equation}\label{vml2}
   \int_{\tilde{t}_0}^{\tilde{t}_0+1}\int\left(m^{-\frac{1}{b}}+\frac{1}{b^2}m^{-\frac{1}{b}-2}m_x^2\right)e^{x_1/L}\,\mathrm{d}x\,\mathrm{d}t
    \lesssim e^{- x_0/L}.
  \end{equation}
\end{corollary}

\begin{proof}
   Note that $\psi_L(y)\geq Ce^{y/L}$ and $\psi_L'(y)\geq Ce^{y/L}$ for $y<0.$ It follows from \eqref{i12} in Lemma \ref{uvm1} and passing to the limit as $x_0\to\infty$ and multiplying by $e^{-x_0/L}$ that
   $$\sup_{t\in\mathbb{R}}\int\left(m^{-\frac{1}{b}}+\frac{1}{b^2}m^{-\frac{1}{b}-2}m_x^2\right)(x+\rho(t),t)e^{x/L}\,\mathrm{d}x\lesssim1.$$
   We similarly note that $m(-x,-t)$ is the solutions of \eqref{bfe}, then the same result for $m(-x,-t)$ also hold for $\delta_0<\frac1L$, which implies that \eqref{vml1} holds.

   Next, we know that $\psi_L'(y)\geq Ce^{y/L}$ for $y<0.$ It follows from \eqref{i12} that
   \begin{align}
     \int_{\tilde{t}_0}^{\tilde{t}_0+1}\int_{x<\rho(t_0)-4b(t-t_0)+x_0}&\left(m^{-\frac{1}{b}}+\frac{1}{b^2}m^{-\frac{1}{b}-2}m_x^2\right)(x,t) e^{(x-\rho(t_0)+4b(t_0-t))/L}\,\mathrm{d}x\,\mathrm{d}t\lesssim1.\label{c3.44}
   \end{align}
We obtain \eqref{vml2} by taking the limit $x_0\to\infty$ and multiplying \eqref{c3.44} by $e^{-x_0/L}$. The proof of this Corollary is completed.
\end{proof}

Proceeding as in the proof of Corollary \ref{cro3.1}, we finally derive the uniform exponential decay of $\varepsilon$ from the monotonicity properties.
\begin{corollary}\label{uvlem3.5}
    Under the assumptions of Lemma \ref{zlem3.1}, for sufficiently small $\epsilon_0$, there exists a positive constant $\delta_0$ such that
  \begin{equation}\label{lcvar}
    \sup_{t\in\mathbb{R}}\int\varepsilon^2(x,t)e^{\delta_0|x|}\,\mathrm{d}x\lesssim1.
  \end{equation}
\end{corollary}

\subsection{Proof of nonlinear Liouville theorem}
In this subsection, we process to prove Theorem \ref{nlt}.
Letting $m_0(x)\in \mathcal{Z}(\mathbb{R})$ be the initial data stated in Theorem \ref{nlt}, we obtain a unique global solution $m(x,t)$ of \eqref{bfe}. Invoking the orbital stability of the unique global solution $m(x,t)$, there exists $\rho(t)\in C^1(\mathbb{R})$ such that
\begin{equation*}
m(x,t)=\gamma Q(x-\rho(t))+\varepsilon(x-\rho(t),t),
\end{equation*}
which satisfies
$$\sup_{t>0}\left\|m(\cdot,t)-\gamma Q(\cdot-\rho(t))\right\|_{H^1(\mathbb{R})}\leq \epsilon,$$
for any $t \in \mathbb{R}$. The function $\varepsilon(x,t)$ satisfies
\begin{equation}\label{vartt}
\varepsilon_t=\frac{1}{1-b}\mathcal{B}(Q)\mathcal{L}\varepsilon+\partial_x\left(\frac{b-1}{2}\left((\partial_x\tilde{\varepsilon})^2-\tilde{\varepsilon}^2\right)-\tilde{\varepsilon}\varepsilon\right)+\rho'(t)\left(\gamma Q'+\partial_x\varepsilon\right),
\end{equation}
where $\tilde{\varepsilon}=\left(1-\partial_x^2\right)^{-1}\varepsilon$, $\mathcal{B}$ and $\mathcal{L}$ are defined in \eqref{ob} and \eqref{ol}.
Let us define two $C^{\infty}-$functions $\phi_M(x)=\phi(x/M)$ and $\Phi_M=b^2Q\left(\frac{x}{M+t^2}\right)$ as in \eqref{vgf}--\eqref{gine}. Relying on \eqref{vartt}, we obtain that, for some $M>0$ (to be chosen later),
 \begin{align}\label{varveq}
    \frac{1}{2}\frac{\mathrm{d}}{\mathrm{d}t}\int\varepsilon^2\phi_M(x)\,\mathrm{d}x
=&\int\varepsilon\phi_M(x)\left(-bq'\varepsilon-q\varepsilon_x-bQ\tilde{\varepsilon}_x-Q'\tilde{\varepsilon}\right)\,\mathrm{d}x\nonumber\\
&+\int\varepsilon\phi_M(x)\partial_x\left(\frac{b-1}{2}\left((\partial_x\tilde{\varepsilon})^2-\tilde{\varepsilon}^2\right)-\tilde{\varepsilon}\varepsilon\right)\,\mathrm{d}x\nonumber\\
&+\rho'(t)\int\varepsilon\phi_M(x)\left(\gamma Q'+\partial_x\varepsilon\right)\,\mathrm{d}x.
\end{align}

Proceeding analogously to Section \ref{s2.2}, to the aim of proving nonlinear Liouville property, we set $w=\varepsilon\sqrt{\Phi_M}$ such that
\begin{align}\label{lw}
    \left(\mathcal{L}w,w\right)=&\frac{2k}{b^2}\int\varepsilon_x^2\Phi_M\alpha\,\mathrm{d}x-\frac{2k(b+1)}{b}\int\varepsilon^2\Phi_M\alpha\,\mathrm{d}x-\frac{k}{b^2}\int\varepsilon^2\Phi_M'\alpha_x\,\mathrm{d}x\nonumber\\
    &-\frac{k}{b^2}\int\varepsilon^2\Phi_M''\alpha\,\mathrm{d}x+\frac{k}{2b^2}\int\varepsilon^2\frac{(\Phi_M')^2}{\Phi_M}\alpha\,\mathrm{d}x.
\end{align}
Similar to \eqref{lzze}, for $b<-1,$ we also claim the following {\bf coercivity property},
\begin{align}\label{lwe}
\left(\mathcal{L}w,w\right)\geq\tilde{\lambda}\int(\varepsilon^2\alpha+\varepsilon_x^2\alpha)\Phi_M\,\mathrm{d}x.
\end{align}
where $\tilde{\lambda}>0$ will be chosen later. The proof of this coercivity property is similar to the previous result in \eqref{lzze}. For this purpose, we first present the following lemma.
\begin{lemma}\label{bvar1}
  For the bilinear form
  \begin{equation}\label{lvar}
   L(\varepsilon,\varepsilon)=\frac{2k}{b^2}\int\varepsilon_x^2\Phi_M\alpha\,\mathrm{d}x-\frac{2k(b+1)}{b}\int\varepsilon^2\Phi_M\alpha\,\mathrm{d}x,
  \end{equation}
  there exists a positive constant $\tilde{\kappa}$ such that
  \begin{equation}\label{lvac}
    L(\varepsilon,\varepsilon)\geq\tilde{\kappa}\int(\varepsilon^2\alpha+\varepsilon_x^2\alpha)\Phi_M\,\mathrm{d}x,
  \end{equation}
where $\varepsilon$ satisfies \eqref{z3.2}--\eqref{z3.7}.
\end{lemma}
\begin{proof}
We compute by using the definition of $L(\varepsilon,\varepsilon)$ in \eqref{lvar} that
 \begin{equation}\label{b1wc}
   L(\varepsilon,\varepsilon)=\left(\mathcal{L}\sqrt{\Phi_M}\varepsilon,\sqrt{\Phi_M}\varepsilon\right)-\int\left(\partial_x\sqrt{\Phi_M}\right)^2\varepsilon^2\alpha\,\mathrm{d}x
   -\int \Phi_M'\varepsilon\varepsilon_x\alpha\,\mathrm{d}x.
 \end{equation}
The proof of Lemma \ref{bvar1} follows the lines of
Lemma \ref{bfg1}. It  remains to check that $\varepsilon\sqrt{\Phi_M}$ satisfies \eqref{cpc} in Lemma \ref{bfgg}.
It follows from the definition of $\Phi_M(x)$ and the properties of $Q$ that
 \begin{equation*}\label{verw}
   \left|\int\sqrt{\Phi_M}\varepsilon\frac{SQ}{\left\|SQ\right\|_{L^2}}\,\mathrm{d}x\right|\leq\theta\left\|\sqrt{\Phi_M}\varepsilon\right\|_{L_{\alpha}^2}
 \end{equation*}
 and
 \begin{equation*}\label{verw2}
   \left|\int\sqrt{\Phi_M}\varepsilon\frac{Q'}{\left\|Q'\right\|_{L^2}}\,\mathrm{d}x\right|\leq\theta\left\|\sqrt{\Phi_M}\varepsilon\right\|_{L_{\alpha}^2},
 \end{equation*}
the above two inequalities imply that
\begin{equation}\label{verw3}
  \left(\mathcal{L}\sqrt{\Phi_M}\varepsilon,\sqrt{\Phi_M}\varepsilon\right)\geq \frac{3\lambda_1}{4}\left\|\sqrt{\Phi_M}\varepsilon\right\|_{H_{\alpha}^1}^2.
\end{equation}
On the other hand, by using the properties of $\Phi_M(x)$ in \eqref{gine}, we have
\begin{equation}\label{verw4}
 \left|\int\left(\partial_x\sqrt{\Phi_M}\right)^2\varepsilon^2\alpha\,\mathrm{d}x\right|\leq\frac{C}{M^2}\int \varepsilon^2\alpha\Phi_M\,\mathrm{d}x\leq\frac{\lambda_1}{8}\int \varepsilon^2\alpha\Phi_M\,\mathrm{d}x,
\end{equation}
and
\begin{equation}\label{verw5}
\left|\int \Phi_M'\varepsilon\varepsilon_x\alpha\,\mathrm{d}x\right|\leq\frac{C}{M^2}\int (\varepsilon^2\alpha+\varepsilon_x^2\alpha)\Phi_M\,\mathrm{d}x\leq\frac{\lambda_1}{8}\int (\varepsilon^2\alpha+\varepsilon_x^2\alpha)\Phi_M\,\mathrm{d}x.
\end{equation}
We conclude gathering \eqref{b1wc} and \eqref{verw3}--\eqref{verw5} that \eqref{lvac} holds if we can choose $\tilde{\kappa}=\lambda_1/2,$ where $\lambda_1$ is given in Lemma \ref{bfg11}. This completes the proof of the Lemma.
\end{proof}

\begin{proof}[Proof of \eqref{lwe}]
To that end, it is still necessary to consider the last three term of \eqref{lw}. For $b<-1,$ it is quite similar to a related proof of \eqref{lzze}.
By using \eqref{vgf}--\eqref{gine} and the choice of $M=\max\left\{1,\frac{32A(1-b)}{\lambda_1(b+1)^2},\frac{32A(1-b)}{\lambda_1}\right\}$ for $b<-1,$  we also obtain
\begin{align}\label{3lw}
  \left|-\frac{k}{b^2}\int\varepsilon^2\Phi_M'\alpha_x\,\mathrm{d}x-\frac{k}{b^2}\int\varepsilon^2\Phi_M''\alpha\,\mathrm{d}x+\frac{k}{2b^2}\int\varepsilon^2\frac{(\Phi_M')^2}{\Phi_M}\alpha\,\mathrm{d}x\right|\leq
 \frac{\lambda_1}{4}\int\varepsilon^2\alpha\Phi_M\,\mathrm{d}x.
\end{align}
 We deduce gathering \eqref{lvac} and \eqref{3lw} that \eqref{lwe} holds for $\tilde{\lambda}=\lambda_1/4$. This completes the proof.
\end{proof}

To the aim of proving nonlinear Liouville property, we claim the following.\\
For $\frac12<s<1$, $b<-1$ and a positive constant $C$, there holds
\begin{align}\label{wlw}
    \left|\frac12\frac{1}{(M+t^2)^{s}}\frac{\mathrm{d}}{\mathrm{d}t}\int \varepsilon^2\phi_M\,\mathrm{d}x+\left(\mathcal{L}w,w\right)\right|\leq \frac{C}{(M+t^2)^{s}}\left\|\varepsilon\right\|^2_{H^1_{\alpha}}.
\end{align}
Indeed, for $\lambda_1>0,$  we deduce from \eqref{varveq} and \eqref{z3.6} that
\begin{align}\label{IW1}
  \left|(-b)\int \varepsilon^2\phi_M q'\,\mathrm{d}x \right|&+\left|\int \varepsilon\varepsilon_x\phi_M q\,\mathrm{d}x \right|+\left|(-b)\int \varepsilon\tilde{\varepsilon}_x\phi_M Q\,\mathrm{d}x \right|+\left|\int \varepsilon\tilde{\varepsilon}\phi_M Q'\,\mathrm{d}x \right|\nonumber\\
  &\leq C\left(\left\|\varepsilon\right\|_{L^{\infty}}+\left\|\tilde{\varepsilon}\right\|_{L^{\infty}}+
  \left\|\tilde{\varepsilon}_x\right\|_{L^{\infty}}\right)\left\|\varepsilon\right\|_{H^1_{\alpha}}\leq K_0\epsilon\left\|\varepsilon\right\|_{H^1_{\alpha}},
\end{align}
 In addition, we also get that
\begin{align}\label{IW2}
  &\left|(b-1)\int\varepsilon\tilde{\varepsilon}_x\tilde{\varepsilon}_{xx}\phi_M\,\mathrm{d}x\right|+\left|(b-1)\int\varepsilon\tilde{\varepsilon}\tilde{\varepsilon}_{x}\phi_M\,\mathrm{d}x\right|
  +\left|-\int\varepsilon^2\tilde{\varepsilon}_x\phi_M\,\mathrm{d}x-\int\varepsilon\varepsilon_x\tilde{\varepsilon}\phi_M\,\mathrm{d}x\right|\nonumber\\
  &\leq\left|(b-1)\int\varepsilon\tilde{\varepsilon}_x\tilde{\varepsilon}_{xx}\phi_M\,\mathrm{d}x\right|+\left|(b-1)\int\varepsilon\tilde{\varepsilon}\tilde{\varepsilon}_{x}\phi_M\,\mathrm{d}x\right|
  +\left|-\frac12\int\varepsilon^2\tilde{\varepsilon}_x\phi_M\,\mathrm{d}x+\frac{1}{2}\int\varepsilon^2\tilde{\varepsilon}\phi_M'\,\mathrm{d}x\right|\nonumber\\
  &\leq C\left(\left\|\varepsilon\right\|_{L^{\infty}}+\left\|\tilde{\varepsilon}\right\|_{L^{\infty}}+
  \left\|\tilde{\varepsilon}_x\right\|_{L^{\infty}}\right)\left\|\varepsilon\right\|_{H^1_{\alpha}}\leq K_0\epsilon\left\|\varepsilon\right\|^2_{H^1_{\alpha}},
\end{align}
and
\begin{align}\label{IW3}
  |\rho'(t)|\left|\int  \varepsilon\phi_M\left(\gamma Q'+\varepsilon_x\right)\,\mathrm{d}x\right|=& |\rho'(t)|\left|\int\gamma Q'\varepsilon\phi_M\,\mathrm{d}x -\frac12\int  \varepsilon^2\phi_M'\,\mathrm{d}x\right|
  \leq K_0\epsilon\left\|\varepsilon\right\|^2_{H_{\alpha}^1}.
\end{align}
Gathering \eqref{varveq}, \eqref{lw}, \eqref{3lw} and \eqref{IW1}--\eqref{IW3}, we conclude the proof of \eqref{wlw} .

Since \eqref{lwe} and \eqref{wlw} are proven, we are in position of showing Theorem \ref{nlt}.
\begin{proof}[Proof of Theorem \ref{nlt}]
Applying \eqref{lwe} and \eqref{wlw}, we obtain that, for a positive constant $C_{b,\lambda_1}$,
\begin{equation}\label{VAR}
  -\frac12\frac{1}{(M+t^2)^{s}}\frac{\mathrm{d}}{\mathrm{d}t}\int \varepsilon^2\phi_M\,\mathrm{d}x\geq\frac{\lambda_1}{4}\int(\varepsilon^2\alpha+\varepsilon_x^2\alpha)\Phi_M\,\mathrm{d}x-\frac{C_{b,\lambda_1}}{(M+t^2)^{s}}.
\end{equation}
provided $\epsilon_0$ is chosen to be sufficiently small.
From the boundedness of $\phi(x),$ we thus have, by integrating \eqref{VAR} on $\mathbb{R},$
\begin{equation}\label{nlt1}
  \int_{\mathbb{R}}\int \varepsilon^2(x,t)\Phi_M(x)\,\mathrm{d}x\,\mathrm{d}t\leq C_{b,M,\lambda_1}\sup_t\|\varepsilon(\cdot,t)\|_{L^2}<+\infty.
\end{equation}
Arguing as in \eqref{c2.73}--\eqref{c2.79}, we deduce from \eqref{lcvar} that $\varepsilon(x,t)=0$ for $x,t\in\mathbb{R}.$ Observing the definition of $\varepsilon$ in \eqref{z3.3} and \eqref{vartt}, we obtain that
$$\varepsilon(x,t)=0,~~m(x+\rho(t),t)=\gamma Q(x).$$
In addition, we derive from \eqref{z3.7} that $\rho(t)=\rho(0),$
which completes the proof of \eqref{nltt}. Therefore,  the proof of Theorem \ref{nlt} is complete.
\end{proof}

\section{Proof of asymptotic stability}\label{s4}
We now focus our  attention in the last  section  to the proof of Theorem \ref{as} for $b<-1$. The core objective is to show that any global solution starting sufficiently close to a lefton (in the function space $\mathcal{Z}(\R)$) will converge to a lefton (up to translation and scaling) as time tends to infinity. We start by leveraging the orbital stability result (from prior work) and the modulation argument from Section \ref{s3} to ensure the solution remains close to a lefton for all time. We then analyze the long-time behavior using a sequence of time tending to infinity: by extracting a subsequence of the solution (aligned via the modulation parameter), we show this subsequence converges to a limit object in a Sobolev space. We prove this limit object is a solution of the $b$-family of equations that satisfies the localized conditions of the nonlinear Liouville Theorem, hence the limit is a lefton. The key reliance for us to prove the asymptotic stability is smoothness and
rigidity properties of the solutions stated in Section \ref{s3}. Our proof strategy is inspired
by \cite{MM3}, which involve compactness of limit object and compact decay conditions.

Let $m_0$ be the initial data stated in Theorem \ref{as}, we obtain a unique solution $m(x, t)$ of \eqref{bfe} by
global well-posedness result. Relying on the stability results, we have
\begin{equation*}
\inf_{y\in\mathbb{R}}\left\|m(\cdot,t)-\gamma Q(\cdot-y(t))\right\|_{H^1(\mathbb{R})}\leq C\epsilon,
\end{equation*}
for any $t>0$.
We deduce from Lemma \ref{zlem3.1} that there exists $\rho(t)\in C^1(\mathbb{R})$ such that
\begin{align}\label{mb}
\|m(\cdot+\rho(t), t)\|_{H^1(\mathbb{R})} \leq\left\|m(\cdot+\rho(t), t)-\gamma Q\right\|_{H^1(\mathbb{R})}+|\gamma|\left\|Q\right\|_{H^1(\mathbb{R})} \lesssim K_0 +  \epsilon
\end{align} for some fixed positive constant $ K_0>0$.
For a time sequence $t_n$ with $t_n\to\infty~(n\to\infty),$ the $H^1$ boundedness of $m(\cdot+\rho(t_n), t_n)$ implies that there exists a time subsequence (still denoted by ${t_n}$) and $m_0^*\in\mathcal{Z}(\mathbb{R})$ such that
\begin{align}\label{mb2}
m(\cdot+\rho(t_n), t_n)\rightharpoonup m_0^*\quad\text{in}\quad H^1(\mathbb{R}).
\end{align}
It follows from \eqref{z3.3} and \eqref{z3.6} that there exists a time sequence $(t_n)_{n\in\mathbb{N}}$ and $\varepsilon_0^*\in \mathcal{Z}(\mathbb{R})$ such that
\begin{equation}\label{iweak}
\varepsilon\left(\cdot, t_n\right)=m(\cdot+\rho(t_n), t_n)-\gamma Q\rightharpoonup \varepsilon_0^* \text { in } H^1_{\alpha}(\mathbb{R}),
\end{equation}
where $t_n\to\infty$ as $n\to\infty.$
In particular, the solution sequence $ m(\cdot+\rho(t_n), t_n)$ converges to $m_0^*(x)$ in $H^1(\mathbb{R})$  as $n\to\infty.$
It is clear that $m^*(x,t)$ is the solution of \eqref{bfe} with the initial data $m_0^*.$
By employing the stability conclusion \eqref{z3.2}, we deduce that $\|m_0^*-Q\|_{H^1(\mathbb{R})}\leq C\epsilon$ and
\begin{align}\label{stbm}
  \|m^*(\cdot,t)-\gamma Q(\cdot-\rho^*(t))\|_{H^1(\mathbb{R})}\leq C\epsilon,
\end{align}
where $\rho^*$ is the modulation parameter corresponding to the decomposition of $m^*$ stated in Lemma \ref{zlem3.1}.

Now we are going to present some technical lemmas. The aim of the following lemmas is to verify that the limit object satisfies the nonlinear Liouville property. We start with a monotonicity argument for $m(x,t)$.
\begin{lemma}\label{male}
  Let $\psi_L$ be defined as in \eqref{tf}--\eqref{psi} for $L\geq4$. Thus, we have for $x_0>0$ and $t\geq0$,
  \begin{align}\label{ma1}
    \limsup_{t\to\infty}\int \left(m^2+m_x^2\right)(x+\rho(t),t)\psi_L(x-x_0)\,\mathrm{d}x
    \lesssim e^{-x_0/L},
  \end{align}
  and
 \begin{align}\label{ma2}
    \limsup_{t\to\infty}\int \left(m^{-\frac1b}+m^{-\frac1b-2}m_x^2\right)(x+\rho(t),t)\psi_L(x-x_0)\,\mathrm{d}x
    \lesssim e^{-x_0/L}.
  \end{align}
\end{lemma}
\begin{proof}
  In view of Lemma \ref{uvm1}, we obtain that
  \begin{equation}\label{c4.4}
    I_{x_0,t}(t)-I_{x_0,t}(0)\lesssim e^{-x_0/L},~~J_{x_0,t}(t)-J_{x_0,t}(0)\lesssim e^{-x_0/L}.
  \end{equation}
   Next, we consider $I_{x_0,t}(0).$ We distinguish now two cases: $x\geq R_0$ and $x<R_0,$ where $R_0$ is a positive constant. On the one hand, relying on the $F_2-$conservation of $m$ and the properties of $\psi_L$, we get that
   \begin{align}\label{rm1}
    \int_{x\geq R_0}m^{-\frac1b}\psi_L(x-\rho(t)-4bt-x_0)\,\mathrm{d}x\leq C\int_{x\geq R_0}m_0^{-\frac{1}{b}}\,\mathrm{d}x\leq \epsilon.
   \end{align}
    On the other hand, for the case $x<R_0,$ we derive from \eqref{z3.7} that
  \begin{align}\label{rm2}
    \int m^{-\frac{1}{b}}\psi_L(x_1)\,\mathrm{d}x\leq\psi_L\left(R_0-\rho(t)-4bt-x_0\right)
  \int m^{-\frac{1}{b}}\,\mathrm{d}x\leq e^{-x_0/L},
  \end{align}
which together with \eqref{rm1} yields $\lim\sup_{t\to\infty}I_{x_0,t}(0)\lesssim e^{-x_0/L}.$ Thus, relying on the $F_2-$conservation of $m,$  the properties of $\psi_L$ and  \eqref{z3.2}, we have
\begin{align}\label{rm3}
\int m^{2}\psi_L(x-\rho(t)-x_0)\,\mathrm{d}x=& \int m^{-\frac{1}{b}}m^{\frac{1}{b}+2}\psi_L(x-\rho(t)-x_0)\,\mathrm{d}x\nonumber\\
    \leq&\left(\|m-\gamma Q\|_{L^{\infty}}+\|\gamma Q\|_{L^{\infty}}\right)^{\frac1b+2}\int m^{-\frac{1}{b}}\psi_L(x-\rho(t)-x_0)\,\mathrm{d}x\nonumber\\
    \leq& K_0 e^{-x_0/L},
\end{align}
    for $\frac1b+2>0.$
Similarly, we derive that  $\limsup_{t\to\infty}J_{x_0,t}(0)\lesssim e^{-x_0/L}$ for $x\in\mathbb{R}.$
We deduce from \eqref{c4.4} and \eqref{rm3} that \eqref{ma1} and \eqref{ma2} hold, which are desired results in Lemma \ref{male}.
\end{proof}

To establish the proof of asymptotic stability, our primary issue lies in demonstrating the weak convergence of the solution sequence. Thus, we construct the following limit object. In the rest of this section, {\bf we write $\mathfrak{m}(x,t)$, $\mathfrak{m}^*(x,t)$ and  $\mathfrak{m}_0^*(x)$ to denote $ m^{-\frac{1}{2b}}(x,t),$ $(m^*)^{-\frac{1}{2b}}(x,t)$ and $(m_0^*)^{-\frac{1}{2b}}(x)$, respectively}.
\begin{lemma} \label {lem-1}
For $A_0>0,$ we have
  \begin{equation}\label{cgl2}
   \mathfrak{m}(\cdot+\rho(t_n),t_n)\to \mathfrak{m}_0^*~~\text{in}~~L^2(x>-A_0),
  \end{equation}
  as $n\to\infty.$
\end{lemma}
\begin{proof}
  For the domain $x>-A_0,$ we distinguish two cases: $x\geq R_0$ and $-A_0<x<R_0$ for $R_0>0.$ In the case $x\geq R_0,$ relying on Lemma \ref{male}, we have that $\|\mathfrak{m}_0^*\|_{L^2(x\geq R_0)}\leq\epsilon$ and
  \begin{equation}\label{cgl21}
  \limsup_{n\to\infty}\left\| \mathfrak{m}(\cdot+\rho(t_n),t_n)\right\|_{L^2(x\geq R_0)}\leq\epsilon.
  \end{equation}
For the case $-A_0<x<R_0$, we define a domain $\mathcal{D}=\{x\in\mathbb{R}:-A_0<x<R_0\}.$ It is obvious that the embedding $\mathcal{Z}(\mathcal{D})\hookrightarrow L^2(\mathcal{D})$ is compact, and \eqref{iweak} gives that
 \begin{equation}\label{cgl22}
\lim_{n\to\infty}\left\| \mathfrak{m}(x+\rho(t_n),t_n)- \mathfrak{m}_0^*\right\|_{L^2(\mathcal{D})}\to0.
\end{equation}

Then, we get from gathering \eqref{cgl21}--\eqref{cgl22} that \eqref{cgl2} holds. This completes the proof of Lemma \ref{lem-1}.
\end{proof}

We now establish the exponential decay of $m^*$ on a finite time intervals.
\begin{lemma}\label{dmr}
  For any $t\geq0, t_0\geq0$ and fixed $L\geq4$, there exists a positive constant $K(t_0)$ such that
  \begin{align}\label{dmr1}
    \int\left((m_0^*)^2+(\partial_xm_0^*)^2\right)(x)\psi_L(x-x_0)\,\mathrm{d}x\lesssim e^{-x_0/L}
    \end{align}
and
\begin{align}\label{dmr2}
    \sup_{t\in[0,t_0]}\int\left((m^*)^2+(\partial_xm^*)^2\right)(x,t)e^{x/L}\,\mathrm{d}x\leq K(t_0).
    \end{align}
\end{lemma}
\begin{proof}
  In view of \eqref{mb2}, we get that
  \begin{align}\label{dmr3}
    m(\cdot+\rho(t_n), t_n)\sqrt{\psi_L(\cdot-x_0)}\rightharpoonup m_0^*\sqrt{\psi_L(\cdot-x_0)}\quad\text{in}\quad H^1(\mathbb{R}),
  \end{align}
  which leads to
  \begin{equation}\label{dmr4}
\left\|m_0^* \sqrt{\psi_L\left(\cdot-x_0\right)}\right\|_{H^1} \leq \liminf _{n \rightarrow+\infty}\left\|m\left(\cdot+\rho\left(t_n\right), t_n\right) \sqrt{\psi_L\left(\cdot-x_0\right)}\right\|_{H^1}.
\end{equation}
We deduce from \eqref{ma1} and \eqref{dmr4} that \eqref{dmr1} holds.

Next, we prove that \eqref{dmr2} also holds. For fixed $t_0>0$ and $x_0>0,$ integration by parts gives that
 \begin{align}\label{dmr5}
   \frac{\mathrm{d}}{\mathrm{d}t}\int (m^*)^2\psi_L(x-x_0)\,\mathrm{d}x=&2\int m^*\psi_L(x-x_0)\left(-u^*\partial_xm^*-b\partial_xu^*m^*\right)\,\mathrm{d}x\nonumber\\
   =&\int u^*(m^*)^2\psi_L'(x-x_0)\,\mathrm{d}x+(1-2b)\int \partial_xu^*(m^*)^2\psi_L(x-x_0)\,\mathrm{d}x.
 \end{align}
It reveals from \eqref{stbm} that
\begin{align}\label{dmr6}
  \left|\int u^*(m^*)^2\psi_L'(x-x_0)\,\mathrm{d}x\right|\leq& \left(\|u^*-\gamma q\|_{L^{\infty}}+|\gamma q|\right)\int (m^*)^2\psi_L'(x-x_0)\,\mathrm{d}x\nonumber\\
  \leq&K_0\int (m^*)^2\psi_L(x-x_0)\,\mathrm{d}x,
\end{align}
since $\psi_L'(y)\lesssim\psi_L(y)$ for $y\in\mathbb{R}.$ It remains to consider the last term on the right hand of \eqref{dmr5}. We derive from \eqref{stbm} that
\begin{align}\label{dmr7}
  \left|\int \partial_xu^*(m^*)^2\psi_L(x-x_0)\,\mathrm{d}x\right|\leq K_0\int (m^*)^2\psi_L(x-x_0)\,\mathrm{d}x.
\end{align}
Combining  \eqref{dmr6} and \eqref{dmr7}, we obtain that
\begin{align}\label{dmrin}
  \frac{\mathrm{d}}{\mathrm{d}t}\int (m^*)^2\psi_L(x-x_0)\,\mathrm{d}x\leq C\int (m^*)^2\psi_L(x-x_0)\,\mathrm{d}x.
\end{align}

In the following, one deduces from \eqref{bfe} and \eqref{stbm} that
\begin{align}\label{dmrx}
   \frac{\mathrm{d}}{\mathrm{d}t}\int (m_x^*)^2\psi_L(x-x_0)\,\mathrm{d}x=&2\int m_x^*\psi_L(x-x_0)\left(-u^*\partial_x^2m^*-b\partial_x^2u^*m^*-(b+1)\partial_xu^*\partial_xm^*\right)\,\mathrm{d}x\nonumber\\
   \leq&K_0\int (m_x^*)^2\psi_L(x-x_0)\,\mathrm{d}x,
 \end{align}
 which together with \eqref{dmrin} yield that
 \begin{align}\label{dmrx1}
   \frac{\mathrm{d}}{\mathrm{d}t}\int \left((m^*)^2+(m_x^*)^2\right)\psi_L(x-x_0)\,\mathrm{d}x
   \leq C\int \left((m^*)^2+(m_x^*)^2\right)\psi_L(x-x_0)\,\mathrm{d}x.
 \end{align}
 Hence, the Gronwall inequality and \eqref{dmr1} imply that
\begin{align}\label{dmrru}
    \sup_{t\in[0,t_0]}&\int\left((m^*)^2+(m_x^*)^2\right)(x,t)\psi_L(x-x_0)\,\mathrm{d}x\nonumber\\
    &\leq \int\left((m^*)^2+(m_x^*)^2\right)(x,0)\psi_L(x-x_0)\,\mathrm{d}x\lesssim e^{-x_0/L}.
    \end{align}
Since $\psi_L(x-x_0)\geq e^{(x-x_0)/L}$ for $x<x_0,$ we obtain
$$\sup_{t\in[0,t_0]}\int\left((m^*)^2+(m_x^*)^2\right)(x,t)e^{(x-x_0)/L}\,\mathrm{d}x\lesssim e^{-x_0/L},$$
which completes the proof of \eqref{dmr2}. The proof of Lemma \ref{dmr} is completed.
\end{proof}

In the following proposition, we establish the weak continuity of the $b-$family flow in a certain appropriate sense of weak convergence.

\begin{proposition}\label{thmwc}
  Given a positive sequence $\left(m_{n, 0}\right)_{n \in \mathbb{N}}$ and a positive function $m_0$ such that
$
m_{n, 0} \rightharpoonup m_0
$
in $H^1(\mathbb{R})$ as $n \rightarrow+\infty$. Let $m_n$ be the unique global solutions to the b-family equation with initial data $m_{n, 0}$, and let $m$ be the unique global solutions to the b-family equation with initial data $m_0$. Then we have, for any $t>0,$
$$
m_n(\cdot, t) \rightharpoonup m(\cdot, t)~~\text { in }~~H^1(\mathbb{R}),
$$
as $n \rightarrow+\infty$.
\end{proposition}
\begin{proof}
  In view of \eqref{bfe} and integration by parts, there exists a positive constant $C$ such that
  \begin{align}\label{a1}
   \frac{\mathrm{d}}{\mathrm{d}t}\left\|m_n\right\|_{H^1}^2=&(1-b)\int \partial_xu_nm_n^2\,\mathrm{d}x+(-1-2b)\int \partial_xu_n(\partial_xm_n)^2\,\mathrm{d}x\nonumber\\
   \leq&C\left\|m_n\right\|_{H^1}^2.
  \end{align}
where $u_n=(1-\partial_x^2)^{-1}m_n.$ We deduce from Gronwall inequality that $\left\|m_n\right\|_{H^1}$ remains bounded for all $t > 0$.
Moreover, the sequence $\{\partial_xm_n\}$ is bounded in $L^2(\mathbb{R}).$ Therefore, there exists $\tilde{m}\in C(\mathbb{R};H^1(\mathbb{R}))$ and a subsequence $\{n_k\}_{k\geq1}$ such that for any $T>0,$
\begin{align}\label{a2}
  m_{n_k}\rightharpoonup\tilde{m}\in C(\mathbb{R}_+;H^1(\mathbb{R}))~~\text{as}~~n\to+\infty.
  \end{align}
  Next, we check that $\tilde{m}$ is a solution of the $b-$family equation. In view of the boundness of $m_n$ and $\partial_xm_n,$ and applying Rellich-Kondrachov theorem, we have that the sequence $\{m_n\}$ are relatively compact in $L^2(\mathbb{R})$ for any $t\in[0,T].$ In addition, for the solution sequence $m_n,$ its time derivative $\partial_tm_n\in C([0,T];L^2(\mathbb{R}))$ satisfies
  \begin{align*}
    \LN\partial_tm_n\RN_{L^2}\leq&\LN u_n\RN_{L^{\infty}}\LN\partial_xm_n\RN_{L^2}+\LN \partial_xu_n\RN_{L^{\infty}}\LN m_n\RN_{L^2}\\
    \leq&\LN m_n\RN_{H^1}\leq C.
  \end{align*}
  As a consequence, $m_n$ are equicontinuous in $C([0,T];L^2(\mathbb{R}))$. Invoking the compactness argument and the Arzela-Ascoli theorem, there exists a subsequence $\{n_k\}$ such that
  \begin{align}\label{a3}
  m_{n_k}\rightharpoonup\tilde{m}_1\in C(\mathbb{R}_+;L^2(\mathbb{R}))~~\text{as}~~n\to+\infty,
  \end{align}
and $\tilde{m}_1$ also is a solution to the $b-$family equation. The uniqueness of the Cauchy probelm in $H^1(\mathbb{R})$ ensures that $\tilde{m}_1=\tilde{m}.$
\end{proof}

\begin{lemma}\label{zlem4.3}
  For a positive constant $A_0>0,$ we have that
\begin{equation}\label{wh1}
  \mathfrak{m}(\cdot+\rho(t_n),t_n+t)\rightharpoonup \mathfrak{m}^*(\cdot,t)~~\text{in}~~L^2(\mathbb{R}),
\end{equation}
\begin{equation}\label{sl2}
  \mathfrak{m}(\cdot+\rho(t_n),t_n+t)\to \mathfrak{m}^*(\cdot,t)~~\text{in}~~L^2(x>-A_0)
\end{equation}
and
\begin{equation}\label{sl2r}
  \rho(t_n+t)-\rho(t_n)\to \rho^*(t),
\end{equation}
for $t\in\mathbb{R},$ and $\rho^*(t)$ is the modulation parameter corresponding to the decomposition of $m^*$ stated in Lemma \ref{zlem3.1}.
\end{lemma}
\begin{proof}
Let us define $\omega_n(x,t)=\mathfrak{m}(x+\rho(t_n),t_n+t)-\mathfrak{m}^*(x,t).$ It follows from \eqref{cgl2} that
  \begin{equation}\label{sl3}
    \int\omega_n^2(x,0)\psi_L(x)\,\mathrm{d}x\to0,
  \end{equation}
as $n\to\infty.$
  Moreover, we now compute \eqref{bfe}, and thereby one obtains that
  \begin{equation}\label{sl4}
    \partial_t\omega_n=-u_n\partial_x\omega_n-\partial_x\mathfrak{m}^*\left(1-\partial_x^2\right)^{-1}\tilde{\omega}_n+\frac12\partial_xu_n\omega_n
    +\frac12\mathfrak{m}^*\left(1-\partial_x^2\right)^{-1}\partial_x\tilde{\omega}_n,
  \end{equation}
where $u_n=(1-\partial_x^2)^{-1}m(x+\rho(t_n),t_n+t),~\tilde{\omega}_n=m(x+\rho(t_n),t_n+t)-m^*(x,t).$
It reveals from \eqref{sl4} that
\begin{align}\label{sl8}
  \frac{\mathrm{d}}{\mathrm{d}t}\int \omega_{n}^2\psi_L\,\mathrm{d}x=&2\int\omega_n\partial_t\omega_n\psi_L\,\mathrm{d}x\nonumber\\
  =&2\int\omega_n\psi_L\bigg(-u_n\partial_x\omega_n-\partial_x\mathfrak{m}^*\left(1-\partial_x^2\right)^{-1}\tilde{\omega}_n\nonumber\\
  &+\frac12\partial_xu_n\omega_n
    +\frac12\mathfrak{m}^*\left(1-\partial_x^2\right)^{-1}\partial_x\tilde{\omega}_n\bigg)\,\mathrm{d}x,
\end{align}
where $\psi_L$ is defined in \eqref{tf} and $L\geq4.$ We deduce from \eqref{z3.6} in Lemma \ref{zlem3.1} that
\begin{align}\label{sl9}
 \left|\int -u_n\omega_n\partial_x\omega_{n}\psi_L\,\mathrm{d}x\right|=&\frac12\left|\int u_n\omega_n^2\psi_L'\,\mathrm{d}x+\int \partial_xu_n\omega_n^2\psi_L\,\mathrm{d}x\right|\nonumber\\
 \lesssim& \left(\|u_n-\gamma q\|_{L^{\infty}}+|\gamma q|+\|\partial_xu_n-\gamma q'\|_{L^{\infty}}+|\gamma q'|\right)\int \omega_{n}^2\psi_L'\,\mathrm{d}x.
\end{align}
 Relying on the \eqref{z3.6}, the Cauchy-Schwarz inequality and Sobolev embedding theorem, we obtain that
\begin{align}\label{sl10}
 &\left|\int \partial_xu_n\omega_{n}^2\psi_L\,\mathrm{d}x+\int \mathfrak{m}^*\omega_{n}(1-\partial_x^2)^{-1}\partial_x\tilde{\omega}_n\psi_L\,\mathrm{d}x\right|\nonumber\\
 &\lesssim\left(\|\partial_xu_n-\gamma q'\|_{L^{\infty}}+|\gamma q'|\right)\left\|\omega_n\sqrt{\psi_L}\right\|_{L^2}+
\left(\|m^*-\gamma Q\|_{L^{\infty}}^{-\frac{1}{b}}+|\gamma Q|^{-\frac{1}{b}}\right)\left\|\omega_n\sqrt{\psi_L}\right\|_{L^2}
 \left\|\partial_xv_n\sqrt{\psi_L}\right\|_{L^2}\nonumber\\
 &\lesssim\int\omega_n^2\psi_L\,\mathrm{d}x.
\end{align}
 where $\|\partial_xv_n\sqrt{\psi_L}\|_{L^2}\lesssim\|\tilde{\omega}_n\sqrt{\psi_L}\|_{L^2}\lesssim\left\|\omega_n\sqrt{\psi_L}\right\|_{L^2}.$
 For the remaining term in \eqref{sl8}, we deduce from Lemma \ref{dmr}, and \eqref{z3.6}, the Cauchy-Schwarz inequality and the decay properties of the lefton that
\begin{align}\label{sl11}
 \left|\int \partial_x\mathfrak{m}^*\omega_{n}(1-\partial_x^2)^{-1}\tilde{\omega}_n\psi_L\,\mathrm{d}x\right|
 &\lesssim\left\|\partial_x\mathfrak{m}^*\sqrt{\psi_L}\right\|_{L^2}\left\|(1-\partial_x^2)^{-1}\tilde{\omega}_n\right\|_{L^{\infty}}\left\|\omega_n\sqrt{\psi_L}\right\|_{L^2}\nonumber\\
 &\lesssim K_0\epsilon\left\|\partial_xm^*\sqrt{\psi_L}\right\|_{L^2}\int\omega_n^2\psi_L\,\mathrm{d}x\nonumber\\
 &\lesssim\int\omega_n^2\psi_L\,\mathrm{d}x,
\end{align}
where we use the inequality $(f+g)^{-\frac1b-2}\lesssim f^{-\frac1b-2}+g^{-\frac1b-2}$ for $b<-1$ and $f,g>0.$

Gathering \eqref{psi} and \eqref{sl8}--\eqref{sl11} , we get that
\begin{align}\label{sl12}
  \frac{\mathrm{d}}{\mathrm{d}t}\int \omega_{n}^2\psi_L\,\mathrm{d}x\lesssim\int \omega_{n}^2\psi_L\,\mathrm{d}x,
\end{align}
which implies that
\begin{equation}\label{sl6}
  \sup_{t\in[0,t_0]}\int \omega_{n}^2(x,t)\psi_L(x)\,\mathrm{d}x\leq C\int \omega_{n}^2(x,0)\psi_L(x)\,\mathrm{d}x
\end{equation}
holds for any $t_0>0$ by employing the Gronwall inequality.
We deduce from  \eqref{sl3} and \eqref{sl6} that \eqref{sl2} holds for $t\geq0.$

Now, it is sufficient to prove that \eqref{sl2} holds for $t<0.$ For any $t_1<0,$ one deduces from the boundness of $m(\cdot+\rho(t_n),t_n+t_1)$ in $H^1(\mathbb{R})$ that there exists a subsequence $\{t_{n_k}\}$ and $\tilde{m}_0\in H^1(\mathbb{R})$ such that
$m(\cdot+\rho(t_{n_k}),t_{n_k}+t_1)\rightharpoonup\tilde{m}_0$
in $H^1$ as $n_k\to\infty.$
Invoking the well-posedness arguments, one gives a solution $\tilde{m}(x,t)\in C(\mathbb{R};\mathcal{Z}(\mathbb{R}))$ of \eqref{bfe} with the initial data $\tilde{m}_0$.
Arguing as the proof of \eqref{sl2} for $t\geq0,$ we obtain that, for any $A_0>0$ and $t\geq0,$
\begin{equation}\label{sl14}
  \mathfrak{m}(\cdot+\rho(t_{n}),t_{n}+t_1+t)\to\tilde{\mathfrak{m}}(\cdot,t)~~\text{in}~~L^2(x>-A_0),
\end{equation}
 where $\tilde{\mathfrak{m}}=\tilde{m}^{-\frac1b}.$
 Next, setting $t=-t_1$ in \eqref{sl14} and combining it with the fact that $ m(\cdot+\rho(t_{n}),t_{n})\rightharpoonup m_0^*(\cdot,t)$ in $H^1$, we obtain that $\tilde{m}(\cdot,-t_1)=m_0^*.$ The uniqueness reveals that $\tilde{\mathfrak{m}}(\cdot,t-t_1)=\mathfrak{m}^*(\cdot,t)$, which implies that \eqref{sl2} holds for $t\geq t_1.$ Thus, we conclude the proof of \eqref{sl2} for any $t\in\mathbb{R}.$

The weak convergence in $L^2$ in \eqref{wh1} for any $t\geq0$ follows from the uniqueness.
Indeed, it is clear that $\mathfrak{m}_n(x,t)=\mathfrak{m}(\cdot+\rho(t_n),t_n+t)$ is bounded in $L^2$, then there exists a function $\mathfrak{m}_0^*\in L^2$ such that $\mathfrak{m}_n(\cdot,0)\rightharpoonup \mathfrak{m}_0^*$ in $L^2$. Invoking the well-posedness arguments, we know that $m_n(x,t)$ is the solution of \eqref{bfe} with the initial data $m_n(x,0)$ and $m^*(x,t)$ corresponds to the initial data $m_0^*$. There exists a subsequence $\{t_{n_k}\}$
such that $m_{t_{n_k}}\rightharpoonup U_1$ in  $\mathcal{Z}$ and $\mathfrak{m}_{t_{n_k}}\rightharpoonup U_1^{-1/b}$ in $L^2$ as $n_k\to\infty$. In addition, $\partial_t m_n(x,t)$ is bounded in $L^2.$
By employing the compactness argument and the Arzela-Ascoli theorem, we know that $m_{t_{n_k}}\rightharpoonup U_1$ in  $H^s$ for $s<1,$ which ensure that $U_1$ is a solution of \eqref{bfe} in the sense of distributions with the initial data $m_0^*$. By the uniqueness of the Cauchy problem in $\mathcal{Z}$, we know that
$U_1=m^*$ and \eqref{wh1} holds. In addition, the proof of \eqref{wh1} is valid for any $t<0$, just as it is for \eqref{sl2}.

It remains to check the convergence results of $\rho(t)$ in \eqref{sl2r}.  For any $t\in\mathbb{R},$ there exists a subsequence such that
\begin{equation}\label{pc}
\rho\left(t_n+t\right)-\rho\left(t_n\right) \to \tilde{\rho}(t),
\end{equation}
as $n\to\infty.$  By continuity for the weak sequence and taking the limit $n\to\infty$, we have
$$
\begin{aligned}
m\left(\cdot+\rho\left(t_n+t\right), t_n+t\right) \rightharpoonup m^*(\cdot+\tilde{\rho}(t), t) \quad \text { in } H^1(\mathbb{R}),
\end{aligned}
$$
which yields that
$$
\begin{aligned}
\varepsilon_n\left(\cdot, t\right)=&m(\cdot+\rho(t_n+t), t+t_n)-\gamma Q   \\ &\rightharpoonup\tilde{\varepsilon}(\cdot, t):=m^*(\cdot+\tilde{\rho}(t), t)-\gamma Q\quad \text { in } H^1(\mathbb{R}).
\end{aligned}
$$
Moreover, $\varepsilon_n(\cdot,t)$ satisfies the orthogonality conditions
\begin{align*}
 \left(\varepsilon_n,Q\right)_{\alpha}=\left(\varepsilon_n,F_2'(Q)\right)_{\alpha}=0.
\end{align*}
By passing to the limit as $n\to\infty,$ we obtain that the orthogonality conditions \eqref{voc} also hold for $\tilde{\varepsilon}$.
It sufficient to prove that the accumulation points of the sequences $ \rho\left(t_n+t\right)- \rho\left(t_n\right)$ is $\rho^*(t)$.
It follows from \eqref{z3.7} in Lemma \ref{zlem3.1} and the compactness argument that the convergence in \eqref{sl2r} holds.
Relying on the
uniqueness of the modulation parameters stated in Lemma \ref{zlem3.1}, we obtain that $\tilde{\rho}(t)=\rho^*(t),$ which completes the proof of Lemma \ref{zlem4.3}.
\end{proof}

\begin{lemma}\label{cwm}
  For a positive constant $A_0>0,$ we have that
  \begin{equation}\label{cwm1}
  m(\cdot+\rho(t_n),t_n)\to m_0^*(\cdot,t)~~\text{in}~~L^2(x>-A_0),
\end{equation}
\begin{equation}\label{cwm2}
  m(\cdot+\rho(t_n),t_n+t)\rightharpoonup m^*(\cdot,t)~~\text{in}~~L^2(\mathbb{R}),
\end{equation}
and
\begin{equation}\label{cwm3}
  m(\cdot+\rho(t_n),t_n+t)\rightharpoonup m^*(\cdot,t)~~\text{in}~~H^1(\mathbb{R})
\end{equation}
for $t\in\mathbb{R},$ and $\rho^*(t)$ is stated in Lemma \ref{zlem4.3}.
\end{lemma}
\begin{proof}
  The proof of \eqref{cwm1} follows the same lines as that of Lemma \ref{lem-1}: the domain is split into two parts and the conclusion is obtained via the compact embedding theorem; therefore the details are omitted.

  We define that $\tilde{\omega}_n=m(x+\rho(t_n),t_n+t)-m^*(x,t).$ It follows from \eqref{cwm1} that
  \begin{equation}\label{cwm4}
    \int\tilde{\omega}_n^2(x,0)\psi_L(x)\,\mathrm{d}x\to0,
  \end{equation}
as $n\to\infty.$
  Moreover, we now compute \eqref{bfe}, and thereby one obtains that
  \begin{equation}\label{cwm5}
    \partial_t\tilde{\omega}_n=-u_n\partial_x\tilde{\omega}_n-\partial_x m^*\left(1-\partial_x^2\right)^{-1}\tilde{\omega}_n-b\partial_xu_n\tilde{\omega}_n
    -b m^*\left(1-\partial_x^2\right)^{-1}\partial_x\tilde{\omega}_n,
  \end{equation}
where $u_n=(1-\partial_x^2)^{-1}m(x+\rho(t_n),t_n+t).$ It follows from \eqref{cwm5} that
\begin{align}\label{cwm6}
  \frac{\mathrm{d}}{\mathrm{d}t}\int \tilde{\omega}_n^2\psi_L\,\mathrm{d}x=&2\int\tilde{\omega}_n\partial_t\tilde{\omega}_n\psi_L\,\mathrm{d}x\nonumber\\
  =&2\int\tilde{\omega}_n\psi_L\bigg(-u_n\partial_x\tilde{\omega}_n-\partial_x m^*\left(1-\partial_x^2\right)^{-1}\tilde{\omega}_n\nonumber\\
  &-b\partial_xu_n\tilde{\omega}_n
    -b m^*\left(1-\partial_x^2\right)^{-1}\partial_x\tilde{\omega}_n\bigg)\,\mathrm{d}x.
\end{align}
Similar to the Lemma \ref{zlem4.3}, we can obtain
\begin{align}\label{cwm7}
 \left|\int u_n\tilde{\omega}_n\partial_x\tilde{\omega}_n\psi_L\,\mathrm{d}x\right|+&\left|\int \partial_x m^*\tilde{\omega}_{n}(1-\partial_x^2)^{-1}\tilde{\omega}_n\psi_L\,\mathrm{d}x\right|+\left|\int \partial_xu_n\tilde{\omega}_{n}^2\psi_L\,\mathrm{d}x\right|\nonumber\\
 &+\left|\int m^*\tilde{\omega}_{n}(1-\partial_x^2)^{-1}\partial_x\tilde{\omega}_n\psi_L\,\mathrm{d}x\right|
 \lesssim\int \tilde{\omega}_{n}^2\psi_L'\,\mathrm{d}x.
\end{align}
We deduce from \eqref{cwm6}--\eqref{cwm7} that
\begin{align}\label{cwm10}
  \frac{\mathrm{d}}{\mathrm{d}t}\int \tilde{\omega}_{n}^2\psi_L\,\mathrm{d}x\lesssim\int \tilde{\omega}_{n}^2\psi_L\,\mathrm{d}x,
\end{align}
which implies that
\begin{equation}\label{cwm11}
  \sup_{t\in[0,t_0]}\int \tilde{\omega}_{n}^2(x,t)\psi_L(x)\,\mathrm{d}x\leq C\int \tilde{\omega}_{n}^2(x,0)\psi_L(x)\,\mathrm{d}x
\end{equation}
holds for any $t_0>0$ by employing the Gronwall inequality.
We deduce from  \eqref{cwm4} and \eqref{cwm11} that \eqref{cwm2} holds for $t\geq0.$ The weak limit \eqref{cwm3} for $t\geq0$ follows from Proposition \ref{thmwc} and the uniqueness of $H^1$ limit. For the cases of $t<0$ in \eqref{cwm2} and \eqref{cwm3}, the analysis is entirely analogous to the corresponding part of the Lemma \ref{zlem4.3}. Consequently, the convergence results in \eqref{cwm2} and \eqref{cwm3} hold for all $t$.

\end{proof}

We will show the compact exponential decay properties for $m^*(x,t)$ on the right.

\begin{lemma} \label{lem4.5}
  For $L\geq4,$ there exists a positive number $\epsilon$ such that
  \begin{align}\label{ued}
    \int_{x>A_0}(m^{*})^{-\frac1b}(x+\rho^*(t),t)\,\mathrm{d}x\leq \epsilon,
  \end{align}
  for all $A_0>0$, $t\geq0$ and $x_0>0.$
\end{lemma}
\begin{proof}
  For any fixed $x_0>0$ and $t\in\mathbb{R},$ it reveals from \eqref{ma2} in Lemma \ref{male} and the positivity of $m^*$ that
  \begin{align}\label{ued1}
    \limsup_{t\to\infty}\int (m^{*})^{-\frac1b}&(x+\rho(t_n),t+t_n)\nonumber\\
    &\times\psi_L(x+\rho(t_n)-\rho(t+t_n)-x_0)\,\mathrm{d}x
    \lesssim e^{-x_0/L}.
  \end{align}
  We apply \eqref{wh1}--\eqref{sl2r} to obtain
  \begin{align*}
  \mathfrak{m}(x+\rho(t_n),t+t_n)&\sqrt{\psi_L(x+\rho_1(t_n)-\rho_1(t+t_n)-x_0)}\nonumber\\
  &\underset{n\to\infty}{\rightharpoonup}
  \mathfrak{m}^*(x,t)\sqrt{\psi_L(x-\rho^*(t)-x_0)}~~\text{in}~~L^2(\mathbb{R}),
  \end{align*}
  which implies that
  \begin{align}\label{ued2}
  &\left\|\mathfrak{m}^*(x,t)\sqrt{\psi_L(x-\rho^*(t)-x_0)}\right\|_{L^2}\nonumber\\
  &\leq\liminf_{n\to\infty} \left\|\mathfrak{m}(x+\rho(t_n),t+t_n)\sqrt{\psi_L(x+\rho(t_n)-\rho(t+t_n)-x_0)}\right\|_{L^2}.
  \end{align}
It follows from \eqref{ued1} and \eqref{ued2} that
\begin{align}\label{ued3}
  \int (m^{*})^{-\frac1b}(x+\rho^*(t),t)\psi_L(x-x_0)\,\mathrm{d}x\lesssim e^{-x_0/L}.
\end{align}
Thus, \eqref{ued} holds since $\psi_L(x-x_0)\geq e^{(x-x_0)/L}$ for $x-x_0<0.$. This completes proof of Lemma \ref{lem4.5}.
\end{proof}

We now present the exponential decay properties of $m^*(x,t)$ on the left. To this end, we are ready to prove the monotonicity property for $\mathfrak{m}(x,t)$. For $x_0>0$ and $t_0\in\mathbb{R},$ we define the quantities
$$\tilde{I}_{x_0,t_0}(t)=\int m^{-\frac1b}(x,t)\psi_L(\tilde{x}_1)\,\mathrm{d}x\quad\text{and}\quad\tilde{J}_{x_0,t_0}(t)=\frac{1}{b^2}\int m^{-\frac1b-2}m_x^2(x,t)\psi_L(\tilde{x}_1)\,\mathrm{d}x,$$
where $\tilde{x}_1=x-\rho(t)+4b(t-t_0)+x_0$ for $t\leq t_0.$

\begin{lemma}\label{tim1}
  Let $b<-1$ and $m\in C(\mathbb{R},\mathcal{Z}(\mathbb{R}))$ be the solutions of \eqref{bfe} satisfying \eqref{z3.3}--\eqref{z3.7}. For any $x_0>0$ and $0\leq t_1\leq t_2,$ it holds that
  \begin{equation}\label{Itmb}
    \tilde{I}_{x_0,t_1}(t_2)-\tilde{I}_{x_0,t_1}(t_1)\lesssim e^{-x_0/L}.
  \end{equation}
\end{lemma}
\begin{proof}
  Arguing as Lemma \ref{uvm1}, it follows from \eqref{bfe} that
  \begin{align}\label{ti1}
    \frac{\mathrm{d}}{\mathrm{d}t}\left(\tilde{I}(t)+\tilde{J}(t)\right)=&-\int m^{-\frac{1}{b}}u\psi_L'(\tilde{x}_1)\,\mathrm{d}x+\frac{1}{b^2}\int m^{-\frac{1}{b}-2}m_x^2u\psi_L'(\tilde{x}_1)\,\mathrm{d}x\nonumber\\&+\frac{2}{1-b}\int m^{-\frac{1}{b}+1}\psi_L'(\tilde{x}_1)\,\mathrm{d}x
    +(4b-\rho'(t))\int m^{-\frac{1}{b}}\psi_L'(\tilde{x}_1)\,\mathrm{d}x\nonumber\\&+\left(\frac{4}{b}-\rho'(t)\right)\int m^{-\frac{1}{b}-2}m_x^2\psi_L'(\tilde{x}_1)\,\mathrm{d}x,
  \end{align}
where $\epsilon_0$ is chosen to be sufficiently small and $R_0$ is defined in Lemma \ref{uvm1}.

Next, let us consider the first term of the right hand of \eqref{ti1}. \\ For $\int m^{-\frac{1}{b}}u\psi_L'(\tilde{x}_1)\,\mathrm{d}x=\int m^{-\frac{1}{b}}\psi_L'(\tilde{x}_1)\left(u-\gamma q(\cdot-\rho)+\gamma q(\cdot-\rho)\right)\,\mathrm{d}x,$ there are two cases for \\ $|x-\rho(t)|.$ In the case $|x-\rho(t)|\geq R_0,$ invoking the properties of $q(x)$ and \eqref{z3.6}, we obtain, for some $\delta_0>0,$
  $$\left|\int m^{-\frac{1}{b}}\psi_L'(\tilde{x}_1)q(\cdot-\rho)\,\mathrm{d}x\right|\leq Ce^{-\delta_0R_0}\int m^{-\frac{1}{b}}\psi_L'(\tilde{x}_1)\,\mathrm{d}x.$$
  In the case $|x-\rho(t)|\leq R_0,$ we have
  $$|\tilde{x}_1|=\left|x-\rho(t)+4b(t-t_0)+x_0\right|\geq\left|4b(t-t_0)+x_0\right|-|x-\rho(t)|\geq4b(t-t_0)+x_0-R_0,$$
  which leads to
  $$\left|\int m^{-\frac{1}{b}}\psi_L'(\tilde{x}_1)q(\cdot-\rho)\,\mathrm{d}x\right|\leq Ce^{(R_0+4b(t_0-t)-x_0)/L}\int m^{-\frac{1}{b}}\,\mathrm{d}x,$$
since $|\psi'_L(x_1)|\lesssim e^{-|\tilde{x}_1|/L}.$ A similar estimates for the second and third terms of the right hand of \eqref{ti1} also hold. Thus, we deduce from \eqref{z3.6} and \eqref{z3.7} in Lemma \ref{zlem3.1} that
\begin{align}\label{ti2}
    \frac{\mathrm{d}}{\mathrm{d}t}\left(\tilde{I}(t)+\tilde{J}(t)\right)
    \leq& 2b\int m^{-\frac{1}{b}}\psi_L'(\tilde{x}_1)\,\mathrm{d}x+\frac{2}{b}\int m^{-\frac{1}{b}-2}m_x^2\psi_L'(\tilde{x}_1)\,\mathrm{d}x\nonumber\\
    &+Ce^{(R_0+4b(t_0-t)-x_0)/L}.
  \end{align}
Integrating \eqref{ti2} on $(t_1,t_2)$, we obtain that \eqref{Itmb} holds. Letting $u(x,t)=m(-x,-t)$ and $\rho_{u}=-\rho(-t),$ substituting $u(x,t)$ into \eqref{Itmb}, we have
\begin{align*}
  \int m^{-\frac1b}(x,-t_2)&\psi_L\left(-x+\rho(-t_2)+4b(t_2-t_1)+x_0\right)\,\mathrm{d}x\nonumber\\
  \leq&\int m^{-\frac1b}(x,-t_1)\psi_L\left(-x+\rho(-t_1)+x_0\right)\,\mathrm{d}x+e^{-x_0/L},
\end{align*}
which implies that
\begin{align}\label{ti3}
  \int m^{-\frac1b}(x,t_2)&\psi_L\left(x-\rho(t_2)-x_0\right)\,\mathrm{d}x\nonumber\\
  \leq&\int m^{-\frac1b}(x,t_1)\psi_L\left(x-\rho(t_1)+4b(t_1-t_2)-x_0\right)\,\mathrm{d}x+e^{-x_0/L},
\end{align}
since $\psi_L(-x)=1-\psi_L(x)$ and $t_2'=-t_1\geq t_1'=-t_2.$ This completes the proof of Lemma \ref{tim1}.
\end{proof}

\begin{lemma} \label {L3}
  Fixing $L\geq4,$ there exists a positive number $\epsilon$ such that
  \begin{align}\label{uel}
    \int_{x<-A_0}\left(m^{*}\right)^{-\frac1b}(x+\rho^*(t),t)\,\mathrm{d}x\leq \epsilon,
  \end{align}
  for all $A_0>0$, $t\geq0$ and $x_0>0.$
\end{lemma}
\begin{proof}
Suppose that $\tilde{n}\in\mathbb{N}$ and $\tilde{t}\in\mathbb{R}$ are given such that $t_{\tilde{n}}\geq t_n+\tilde{t}$.
We start by employing the monotonicity property \eqref{Itmb} in Lemma \ref{tim1} to obtain that
\begin{align}\label{uel1}
  \int m^{-\frac1b}\left(x, t_{\tilde{n}}\right) &\psi_L\left(x-\rho\left(t_{\tilde{n}}\right)+4b\left(t_{\tilde{n}}-\left(t_{n}+\tilde{t}\right)\right)+x_{0}\right) \,\mathrm{d} x \nonumber\\
& \lesssim \int m^{-\frac1b}\left(x, t_{n}+\tilde{t}\right) \psi_L\left(x-\rho\left(t_{n}+\tilde{t}\right)+x_{0}\right) \mathrm{d} x+e^{-x_{0} / L} .
\end{align}
Moreover, relying on \eqref{wh1} and \eqref{sl2} in Lemma \ref{zlem4.3}, it is inferred  that
\begin{align}\label{uel2}
   \mathfrak{m}\left(x, t_{n}+\tilde{t}\right) &\sqrt{\psi_L\left(x-\rho\left(t_{n}+\tilde{t}\right)+x_{0}\right)}\nonumber\\
   &\rightharpoonup \mathfrak{m}^*(x,\tilde{t})\sqrt{\psi_L\left(x-\rho^*\left(\tilde{t}\right)+x_{0}\right)}~~\text{in}~~L^2,
\end{align}
as $n\to\infty.$ There exists $\tilde{n}_0\in\mathbb{N}$ such that
\begin{align}\label{uel3}
   \int m^{-\frac1b}\left(x, t_{n}+\tilde{t}\right) &\psi_L\left(x-\rho\left(t_{n}+\tilde{t}\right)+x_{0}\right)\,\mathrm{d} x \nonumber\\
   &\lesssim \int (m^*)^{-\frac1b}(x,\tilde{t})\psi_L\left(x-\rho^*\left(\tilde{t}\right)+x_{0}\right)\,\mathrm{d} x +e^{-x_0/L},
\end{align}
for $n\geq \tilde{n}_0.$ It follows from \eqref{uel1} and \eqref{uel3} that
\begin{align}\label{uel4}
   &\int m^{-\frac1b}\left(x, t_{\tilde{n}}\right) \psi_L\left(x-\rho\left(t_{\tilde{n}}\right)+4b\left(t_{\tilde{n}}-\left(t_{n}+\tilde{t}\right)\right)+x_{0}\right) \,\mathrm{d} x \nonumber\\
   &\lesssim \int (m^*)^{-\frac1b}(x,\tilde{t})\psi_L\left(x-\rho^*\left(\tilde{t}\right)+x_{0}\right)\,\mathrm{d} x +e^{-x_0/L}.
\end{align}
In addition, it transpires from \eqref{cgl2} in Lemma \ref{lem-1} that
\begin{align*}
   &\int m^{-\frac1b}\left(x+\rho\left(t_{\tilde{n}}\right), t_{\tilde{n}}\right) \psi_L\left(x+4b\left(t_{\tilde{n}}-\left(t_{n}+\tilde{t}\right)\right)+x_{0}\right) \,\mathrm{d} x
   \to0,
\end{align*}
as $\tilde{n}\to\infty.$ Then, we have
\begin{align}\label{uel5}
\int m^{-\frac1b}\left(x+\rho_{1}\left(t_{\tilde{n}}\right), t_{\tilde{n}}\right) \psi_L\left(x+4b\left(t_{\tilde{n}}-\left(t_{n}+\tilde{t}\right)\right)+x_{0}\right) \,\mathrm{d} x\gtrsim-e^{-x_0/L}.
\end{align}
We derive from \eqref{uel4}--\eqref{uel5} that
\begin{align}\label{uel6}
\int (m^*)^{-\frac1b}(x,t)\left(-\psi_L\left(x-\rho^*\left(t\right)+x_{0}\right)\right)\,\mathrm{d} x \lesssim e^{-x_0/L}.
\end{align}

In addition, invoking \eqref{ti3}, we have
\begin{align}\label{mei}
    \int m^{-\frac1b}\left(x, t_{\tilde{n}}\right) &\psi_L\left(x-\rho\left(t_{\tilde{n}}\right)-x_{0}\right) \,\mathrm{d} x \nonumber\\
& \lesssim \int m^{-\frac1b}\left(x, t_{n}+\tilde{t}\right) \psi_L\left(x-\rho\left(t_{n}+\tilde{t}\right)-4b\left(t_{\tilde{n}}-\left(t_{n}+\tilde{t}\right)\right)-x_{0}\right) \mathrm{d} x+e^{-x_{0} / L} .
\end{align}
Similarly, \eqref{wh1} and \eqref{sl2} in Lemma \ref{zlem4.3} ensure that
\begin{align}\label{mei2}
 \mathfrak{m}\left(x, t_{\tilde{n}}\right) \sqrt{\psi_L\left(x-\rho\left(t_{\tilde{n}}\right)-x_{0}\right)}
\rightharpoonup\mathfrak{m}_0^*\sqrt{\psi_L\left(x-x_{0}\right)}\quad\text{in}\quad L^2(\mathbb{R})
\end{align}
and
\begin{align}\label{mei3}
   \mathfrak{m}\left(x, t_{n}+\tilde{t}\right) \sqrt{\psi_L\left(x-\rho\left(t_{n}+\tilde{t}\right)-4b\left(t_{\tilde{n}}-\left(t_{n}+\tilde{t}\right)\right)-x_{0}\right)}
   \rightharpoonup0,
\end{align}
as $n\to\infty.$ It is deduced from \eqref{mei}--\eqref{mei3} that
$$\int \left(m_0^*\right)^{-\frac1b}\psi_L\left(x-x_{0}\right)\,\mathrm{d}x\lesssim e^{-x_0/L},$$
which together with \eqref{uel6} yields that
\begin{align}\label{meir}
    &\int \left(m_0^*\right)^{-\frac1b}\psi_L\left(x-x_{0}\right)-(m^*)^{-\frac1b}\psi_L\left(x-\rho^*\left(t\right)+x_{0}\right)\,\mathrm{d}x\nonumber\\
    &=\int \left(m_0^*\right)^{-\frac1b}\psi_L\left(x-x_{0}\right)\,\mathrm{d}x-\int(m^*)^{-\frac1b}\,\mathrm{d}x+\int(m^*)^{-\frac1b}\left(1-\psi_L\left(x-\rho^*\left(t\right)+x_{0}\right)\right)\,\mathrm{d}x.
\end{align}
Therefore, we have
\begin{align}\label{meir2}
    \int(m^*)^{-\frac1b}(x+\rho^*\left(t\right),t)\left(1-\psi_L\left(x+x_{0}\right)\right)\,\mathrm{d}x\leq  Ce^{-x_0/L},
\end{align}
where $C$ is a positive constant depending on $F_2(m_0^*)<\infty.$
The proof of \eqref{uel} is finished since $1-\psi_L(x+x_0)\geq\frac12$ for $x\leq-x_0$. This completes the proof of Lemma \ref{L3}.
\end{proof}
We are now in the position to prove the asymptotic stability result.
\begin{proof}[Proof of Theorem \ref{as}]
For any sequence $\{t_{n}\}$,
there exist a subsequence $\{t_{n_k}\}$ with $t_{n_k}\to\infty$ and $m_0^*\in \mathcal{Z}(\mathbb{R})$ such that
\begin{equation}\label{iw1}
m(\cdot+\rho(t_{n_k}),t_{n_k})\rightharpoonup m_0^* \text { in } H^1(\mathbb{R}).
\end{equation}
Thus, the solution $m^*(\cdot,t)$ of \eqref{bfe} with $m^*(\cdot,0)=m_0^*$ satisfying \eqref{stbm}, \eqref{ued}, \eqref{uel} and one also has $\rho^*(0)=0.$
It follows from \eqref{ued}, \eqref{uel} and Theorem \ref{nlt} that there exists $r\in\mathbb{R}$ such that
\begin{align}\label{dd}
m(x,t)=\gamma Q(x-r).
\end{align}
By using the uniqueness of the decomposition in Lemma \ref{zlem3.1}, we obtain that $r=\rho^*(0)=0.$ Hence, we deduce from \eqref{iw1} and \eqref{dd} that $m(\cdot+\rho(t_{n_k}),t_{n_k})-\gamma Q$ weakly converges to $0$ in $H^1(\mathbb{R})$ as $k\to\infty.$ This convergence result holds true for any time sequence $\{t_n\}$ with $t_n\to\infty,$ which yields that
\begin{equation}\label{asymr}
    m(\cdot+\rho(t),t)\underset{t\to\infty}{\rightharpoonup} \gamma Q\quad\text{in}\quad H^1(\mathbb{R}).
\end{equation}

Now, there remains to prove \eqref{asympt lim}. Integrating \eqref{Iuv} on $[t_1,t_2]$, we obtain that
\begin{align}\label{re1}
 \int &\left(m^{-\frac{1}{b}}+\frac{1}{b^2}m^{-\frac{1}{b}-2}m_x^2\right)(x,t_2)\psi_L\left(x-\rho(t_1)+4b(t_2-t_1)
 -x_0\right)\,\mathrm{d}x\nonumber\\
 &\leq\int\left(m^{-\frac{1}{b}}+\frac{1}{b^2}m^{-\frac{1}{b}-2}m_x^2\right)(x,t_1)\psi_L\left(x-\rho(t_1)
 -x_0\right)\,\mathrm{d}x+Ke^{-x_0/L},
\end{align}
for all $x_0>0$, $t_2\geq t_1$. We also note that the decomposition $\varepsilon^2(x,t)=m^2(x,t)-2\gamma\varepsilon(x,t) Q(x-\rho(t))-\gamma^2Q^2(x-\rho(t))$ and
$(\partial_x\varepsilon)^2(x,t)=m_x^2(x,t)-2\gamma\varepsilon_x(x,t) Q'(x-\rho(t))-\gamma^2Q'^2(x-\rho(t))$.
In addition, relying on \eqref{z3.6}, the orthogonality condition in \eqref{voc} and the properties of $Q$ and $\psi_L$, we have
\begin{align}\label{re2}
  \left|\int\left(\varepsilon(x+\rho(t),t)Q+\varepsilon_x(x+\rho(t),t)Q'\right)
  \psi_L(x-x_0)\,\mathrm{d}x\right|\lesssim e^{-x_0/L}.
\end{align}
Therefore,  we deduce from \eqref{z3.6}, \eqref{asymr} and \eqref{re1}--\eqref{re2} that
\begin{align}\label{re4}
 \underset{t\to\infty}{\lim\sup}\int \left(\varepsilon^2+(\partial_x\varepsilon)^2\right)(x,t_2)&\psi_L\left(x-\rho(t_1)+4b(t_2-t_1)
 -x_0\right)\,\mathrm{d}x\lesssim e^{-x_0/L},
\end{align}
where $-1/b-2<0$ for $b<-1.$
We define $\beta t:=\rho(t_1)-4b(t-t_1)+x_0$ for $0<t_1<t$ and $t\to\infty$ as $t_1\to\infty$. Thus, by using \eqref{re4}, we have
\begin{equation*}\label{re7}
   \limsup_{t\to\infty}\int\left(\varepsilon^2+(\partial_x\varepsilon)^2\right)(x,t)\psi_L(x-\beta t)\,\mathrm{d}x
   \lesssim e^{-x_0/L},
 \end{equation*}
and the proof of \eqref{asympt lim} is concluded. This completes the proof of asymptotic result in Theorem \ref{as}.
\end{proof}
\section{Appendix}
\begin{proof}[Proof of the operator \texorpdfstring{$\mathcal{B}(Q)\mathcal{L}$ and $\mathcal{L}\mathcal{B}(Q)$}{B(Q)L and LB(Q)}]
Recalling that
\begin{align*}
 \mathcal{L}&=k\left(-\partial_x\left(\frac{2\alpha}{b^2}\partial_x\right)-\frac{2(b+1)\alpha}{b}\right),\\
   \mathcal{B}&=-\left(b m \partial_x+m_x\right)\left(\partial_x-\partial_x^3\right)^{-1}\left(b m \partial_x+(b-1) m_x\right),
\end{align*}
we get that
\begin{align*}
 \mathcal{L}v&=-\frac{2k}{b^2}\left(\partial_x(\alpha v_x)+b(b+1)\alpha v\right):=-\frac{2k}{b^2}\tilde{\mathcal{L}}v,\\
   \mathcal{B}(Q)&=-\left(b Q \partial_x+Q'\right)\left(\partial_x-\partial_x^3\right)^{-1}\left(b Q \partial_x+(b-1) Q'\right).
\end{align*}
A direct calculation then gives
\begin{align*}
  \left(b Q \partial_x+(b-1) Q'\right)  \tilde{\mathcal{L}}v=&b Q^{-\frac{1}{b}-1}\partial_x^3v-3(b+1)Q^{-\frac{1}{b}-2} Q'\partial_x^2v+b(3b^2+6b+2)Q^{-\frac{1}{b}-1}\partial_xv\\
  &-b(b+1)(b+2)Q^{-\frac{1}{b}-2}Q'v-\frac{b(2b+1)(b+3)}{2k}\partial_xv\\
  =&-b\left(\partial_x-\partial_x^3\right)\left(Q^{-\frac{1}{b}-1}v\right)+\frac{b^3-b^2}{2k}\partial_xv,
\end{align*}
which leads to
\begin{align}\label{ble}
    \mathcal{B}(Q)\mathcal{L}v=-\frac{2k}{b^2}\mathcal{B}(Q)\tilde{\mathcal{L}}v=&2k Q^{-\frac{1}{b}-1}Q'v-2kQ^{-\frac{1}{b}}\partial_xv\nonumber\\
    &+b(b-1)Q\partial_xh+(b-1)Q'h,
\end{align}
where $h:=(1-\partial_x^2)^{-1}v.$ We deduce from
$$q'=\frac{2k}{b(b-1)}Q^{-\frac{1}{b}-1}Q',~~q=\frac{2k}{1-b}Q^{-\frac{1}{b}}$$
that
\begin{align*}
\frac{1}{1-b}\mathcal{B}(Q)\mathcal{L}v=&-bq'v-q\partial_xv-bQ\partial_xh-Q'h\\
=&\partial_x\left(q\partial_x^2h+(b-1)q'\partial_xh-(b+1)qh+q''h\right).
\end{align*}

Next, we compute the operator $\mathcal{L}\mathcal{B}(Q)$. In view of the definition of $\mathcal{B}(Q),$ one defines that
$$f=\left(\partial_x-\partial_x^3\right)^{-1}\left(b Q \partial_xv+(b-1) Q'v\right):=\left(\partial_x-\partial_x^3\right)^{-1}V,$$
which indicates that $\mathcal{B}(Q)v=-bQ\partial_xf-Q'f.$ In addition, a straightforward computation shows that
\begin{align*}
\tilde{\mathcal{L}}\mathcal{B}(Q)v&=-\frac{2b+1}{b}Q^{-\frac{1}{b}-3}Q'\partial_x\left(\mathcal{B}(Q)v\right)+Q^{-\frac{1}{b}-2}\partial_x^2\left(\mathcal{B}(Q)v\right)
+b(b+1)Q^{-\frac{1}{b}-2}\mathcal{B}(Q)v,\\
    \partial_x\left(\mathcal{B}(Q)v\right)&=-bQ\partial_x^2f-(b+1)Q'\partial_xf-Q''f,   \quad {\rm and} \\
    \partial_x^2\left(\mathcal{B}(Q)v\right)&=-bQ\partial_x^3f-(2b+1)Q'\partial_x^2f-(b+2)Q''\partial_x f-Q'''f,
\end{align*}
which transpires from  the fact
$$Q''=b^2Q-\frac{b(3b+1)}{2k}Q^{\frac1b+2},~~(Q')^2=b^2\left(Q^2-\frac{Q^{\frac{1}{b}+3}}{k}\right),$$
that
\begin{align*}
    \tilde{\mathcal{L}}\mathcal{B}(Q)v=&-bQ^{-\frac{1}{b}-1}\partial_x^3f+bQ^{-\frac{1}{b}-1}\partial_xf+\frac{b^2-b^3}{2k}\partial_xf\\
    =&bQ^{-\frac{1}{b}-1}\left(V-\partial_xf\right)+bQ^{-\frac{1}{b}-1}\partial_xf+\frac{b^2-b^3}{2k}\partial_xf\\
    =&bQ^{-\frac{1}{b}-1}V+\frac{b^2-b^3}{2k}(1-\partial_x^2)^{-1}V.
\end{align*}
It then follows that
\begin{align}\label{lbe}
    \mathcal{L}\mathcal{B}(Q)v=-\frac{2k}{b^2}\tilde{\mathcal{L}}\mathcal{B}(Q)v=&-2kQ^{-\frac{1}{b}}\partial_xv+\frac{2k(1-b)}{b}Q^{-\frac{1}{b}-1}Q'v\nonumber\\
    &+(b-1)(1-\partial_x^2)^{-1}\left(bQ\partial_xv+(b-1)Q'v\right).
\end{align}
This completes the proof.
\end{proof}
\begin{proof}[Proof of \eqref{i1}]
We compute from \eqref{bfe} that
\begin{align}\label{vfm1}
    \frac{\mathrm{d}}{\mathrm{d}t}\int m^{-\frac{1}{b}}\psi_L(x_1)\,\mathrm{d}x=&-\frac{1}{b}\int m^{-\frac{1}{b}-1}m_t\psi_L(x_1)\,\mathrm{d}x+4b\int m^{-\frac{1}{b}}\psi_L'(x_1)\,\mathrm{d}x\nonumber\\
    =&\int m^{-\frac{1}{b}}u\psi_L'(x_1)\,\mathrm{d}x+2\int m^{-\frac{1}{b}}u_x\psi_L(x_1)\,\mathrm{d}x+4b\int m^{-\frac{1}{b}}\psi_L'(x_1)\,\mathrm{d}x.
\end{align}
In a similar way, we also have that
\begin{align*}
    \frac{1}{b^2}\frac{\mathrm{d}}{\mathrm{d}t}\int m^{-\frac{1}{b}-2}m_x^2\psi_L(x_1)\,\mathrm{d}x=&-\frac{2b+1}{b^3}\int m^{-\frac{1}{b}-3}m_t m_x^2\psi_L(x_1)\,\mathrm{d}x+\frac{2}{b^2}\int m^{-\frac{1}{b}-2}m_x m_{xt}\psi_L(x_1)\,\mathrm{d}x\nonumber\\
    &+\frac{4}{b}\int m^{-\frac{1}{b}-2}m_x^2\psi_L'(x_1)\,\mathrm{d}x:=\Delta_1+\Delta_2+\frac{4}{b}\int m^{-\frac{1}{b}-2}m_x^2\psi_L'(x_1)\,\mathrm{d}x,
\end{align*}
where
\begin{align}\label{vfm2}
    \Delta_1=&\frac{2b+1}{b^3}\int m^{-\frac{1}{b}-3} m_x^3 u\psi_L(x_1)\,\mathrm{d}x+\frac{2b+1}{b^2}\int m^{-\frac{1}{b}-2} m_x^2 u_x\psi_L(x_1)\,\mathrm{d}x,
\end{align}
and
\begin{align}\label{vfm3}
    \Delta_2=&-\frac{2(b+1)}{b^2}\int m^{-\frac{1}{b}-2} m_x^2 u_x\psi_L(x_1)\,\mathrm{d}x-\frac{2}{b^2}\int m^{-\frac{1}{b}-2} m_x m_{xx}u\psi_L(x_1)\,\mathrm{d}x\nonumber\\
    &-\frac{2}{b}\int m^{-\frac{1}{b}-2} m m_x u_{xx}\psi_L(x_1)\,\mathrm{d}x.
\end{align}
Gathering \eqref{vfm2} and \eqref{vfm3}, we get that
\begin{align}\label{vfm4}
    \frac{1}{b^2}\frac{\mathrm{d}}{\mathrm{d}t}\int m^{-\frac{1}{b}-2}m_x^2\psi_L(x_1)\,\mathrm{d}x=&\frac{1}{b^2}\int m^{-\frac{1}{b}-2} m_x^2 u\psi_L'(x_1)\,\mathrm{d}x-2\int m^{-\frac{1}{b}}  u_x\psi_L(x_1)\,\mathrm{d}x\nonumber\\
    &+\frac{2}{1-b}\int m^{-\frac{1}{b}+1}  \psi_L'(x_1)\,\mathrm{d}x-2\int m^{-\frac{1}{b}}  u\psi_L'(x_1)\,\mathrm{d}x.
\end{align}
Hence, combining \eqref{vfm1} with \eqref{vfm4}, this completes  the proof of \eqref{i1}.
\end{proof}

\section*{Acknowledgments}
 The work of Wang is partially supported by the NSFC grant 11901092. The authors express their gratitude to Professor Luc Molinet for stimulating discussions.

 \section*{Data Availability}
 The data that supports the findings of this study are available within the article.

 \section*{Conflict of interest}
 The authors have no conflicts to disclose.

\end{document}